\documentclass{amsart}
\usepackage{mathptmx}      %
\usepackage{amssymb}
\usepackage{amscd}
\usepackage{amsmath}
\usepackage{amsthm}
\usepackage{enumitem}

\usepackage{mathrsfs}
\usepackage[english]{babel}
\usepackage{cite}
\usepackage{allrunes}
\usepackage{color}
\usepackage{mathtools}

\renewcommand{\R}{\mathbb{R}}

\renewcommand{\T}{\mathbb{T}}
\newcommand{\Z}{\mathbb{Z}}
\newcommand{\dd}{\mathrm{d}}
\newcommand{\sqrbrak}[1]{\left[#1\right]}
\newcommand{\norm}[1]{\left\Vert#1\right\Vert}
\newcommand{\set}[1]{\left\{#1\right\}}
\newcommand{\brkt}[1]{\left(#1\right)}
\newcommand{\sqbrkt}[1]{\left[#1\right]}
\newcommand{\abs}[1]{\left|#1\right|}
\newcommand{\pr}[2]{\langle{#1,#2}\rangle}
\renewcommand{\p}[1]{\langle{#1}\rangle}
\renewcommand{\d}{\partial}
\newcommand{\supp}{\mathop{\rm supp}}
\renewcommand{\S}{C^\infty_0(\R^n)}

\newtheorem{thm}{Theorem}[section] %
\newtheorem{lem}[thm]{Lemma}     %
\newtheorem{cor}[thm]{Corollary}
\newtheorem{prop}[thm]{Proposition}
\newtheorem{definition}[thm]{Definition}
\newtheorem{remark}[thm]{Observation}
\newtheorem{example}[thm]{Example}

 \title[Rough Fourier integral operators]
{Estimates for rough Fourier integral and pseudodifferential operators and applications to the boundedness of multilinear operators}

\author{Salvador Rodr\'iguez-L\'opez} 
\email{salvador@math.uu.se}

\author{Wolfgang Staubach}
\email{wulf@math.uu.se}
\address{Department of Mathematics, Uppsala University, Uppsala, SE 75106,  Sweden}
\thanks{Both authors were supported by  the EPSRC First Grant Scheme reference number EP/H051368/1. The first author is also partially supported by the grant MTM2010-14946.}

\begin{document}
\begin{abstract}
We study the boundedness of rough Fourier integral and pseudodifferential operators, defined by general rough H\"ormander class amplitudes, on Banach and quasi-Banach $L^p$ spaces.  Thereafter we apply the aforementioned boundedness in order to improve on some of the existing boundedness results for H\"ormander class bilinear pseudodifferential operators and certain classes of bilinear (as well as multilinear) Fourier integral operators. For these classes of amplitudes, the boundedness of the aforementioned Fourier integral operators turn out to be sharp. Furthermore we also obtain results for rough multilinear operators.
\end{abstract}
\maketitle
\section{Introduction and summary of the results}
A (linear) Fourier integral operator or FIO for short, is an operator that can be written locally in the form
\begin{align*}
   T_{a}f(x)= (2\pi)^{-n} \int_{\R^n} e^{i\varphi(x,\xi)} a(x,\xi) \widehat{f}(\xi) \, \dd\xi,
\end{align*}
where $a(x,\xi)$ is the {\it{amplitude}}, $\varphi(x,\xi)$ is the {\it{phase function}} and $f$ belongs to $\S$. In case the phase function $\varphi(x,\xi)=\langle x, \xi \rangle$, the Fourier integral operator is called a pseudodifferential operator, which in what follows will be abbreviated as $\Psi$DO. The study of these operators, which are intimately connected to the theory of linear partial differential operators, has a long history. There is a large body of results concerning the regularity, e.g. the $L^p$ boundedness, of FIOs and $\Psi$DOs, but due to the lack of space we only mention those investigations that are of direct relevance to the current paper.

The most widely used class of amplitudes are those introduced by H\"ormander in \cite{Hor}, the so called $S^{m}_{\varrho, \delta}$ class, that consists of $a(x,\xi)\in C^{\infty}(\R^{n} \times \R^n)$ with $$|\partial^{\alpha}_{\xi} \partial^{\beta}_{\xi} a(x,\xi)| \leq C_{\alpha \beta} (1+|\xi|)^{m-\varrho|\alpha|+\delta|\beta|},$$ $m\in \R$, $\varrho, \delta \in[0,1]$.
For phase functions one usually assumes that $\varphi\in C^{\infty}(\R^{n} \times \R^n\setminus 0)$ is homogeneous of degree $1$ in the frequency variable $\xi$ and satisfies the {\it{non-degeneracy condition}}, that is the mixed Hessian matrix $[\frac{\partial^2 \varphi}{\partial x_j\partial \xi_k}]$ has non-vanishing determinant.

 The sharp local $L^p$ ($p\in(1,\infty)$) boundedness of FIOs, under the assumptions of $a(x,\xi)\in S^{m}_{\varrho, 1-\varrho}$ being compactly supported in the spatial variable $x$ and $\varrho\in[\frac{1}{2},1]$, $m=(\varrho-n)|1/p -1/2|$ was established by A. Seeger, C. D. Sogge and E.M. Stein \cite{SSS}. The global $L^p$ boundedness of FIOs (i.e. boundedness without the assumption of compact spatial support of the amplitude) has also been investigated in various contexts and here we would like to mention boundedness of operators with smooth amplitudes in the so called $\mathbf{SG}$ classes, due to E. Cordero, F. Nicola and L. Rodino in \cite{CNR1}; the boundedness of operators with amplitudes in $S^{m}_{1,0}$ on the space of compactly supported distributions whose Fourier transform is in $L^p(\R^{n})$ (i.e. the $\mathcal{F}L^p$ spaces) due to Cordero, Nicola and Rodino in \cite{CNR2} and Nicola's refinement of this investigation in \cite{Nic}; and finally, S. Coriasco and M. Ruzhansky's global $L^{p}$ boundedness of Fourier integral operators \cite{CorRuzh}, with amplitudes that belong to a certain subclass of $S^{0}_{1, 0}.$

In this paper we consider the problem of boundedness of Fourier integral operators with amplitudes that are non-smooth in the spatial variables and exhibit an $L^p$ type behaviour in those variables for $p\in[1, \infty].$ This is a continuation of the investigation of boundedness of rough pseudodifferential operators made by C. Kenig and W. Staubach \cite{KenStaub} and that of regularity of rough Fourier integral operators carried out by D. Dos Santos Ferreira and W. Staubach \cite{DosStaub}, where the boundedness of the aforementioned operators were established under the condition that the corresponding amplitudes are $L^\infty$ functions in the spatial variables.

One motivation for the study of these specific classes of rough oscillatory integrals whose amplitudes have $L^p$ spatial behaviour is, as will be demonstrated in this paper,  its applicability in proving boundedness results for multilinear pseudodifferential and Fourier integral operators.

A study of rough pseudodifferential operators without any regularity assumption in the spatial variables were carried out in \cite{KenStaub} and A. Stefanov's paper \cite{Stefanov}, where the irregularity of the symbols of the operators are of $L^\infty$ type. Prior to these investigations, a systematic study of pseudodifferential operators with limited smoothness was carried out by M. Taylor in \cite{Tay} The corresponding problem for the Fourier integral operators were investigated in \cite{DosStaub}.
In the case of $L^p$ spatial behaviour, an investigation of boundedness of pseudodifferential operators was carried out by N. Michalowski, D. Rule and W. Staubach in \cite{MRS}.
To our knowledge, there has been no investigation of regularity of rough Fourier integral operators with $L^p$ spatial behaviour prior to the one in the present paper.

Turning to the multilinear setting, in our investigation we shall consider {\it{multilinear Fourier integral operators}} of the form
\begin{equation*}
T_a(f_1, \dots, f_N) (x)= \int_{\R^{Nn}} a(x,\xi_1, \dots, \xi_N)\, e^{ i\sum_{j=1}^{N} \varphi_{j}(x, \xi_{j})} \, \prod_{j=1}^{N}\widehat{f}(\xi_j)\, \dd \xi_1\dots \dd \xi_N.
\end{equation*}
The amplitude $a(x,\xi_{1},\dots,\xi_N)$ is usually assumed to be smooth in all the variables, satisfying an estimate of the type
\begin{equation*}
\vert \partial^{\beta}_x \partial_{\xi_{1}}^{\alpha_1}\dots\partial_{\xi_{N}}^{\alpha_N}a(x,\xi_{1},\dots,\xi_N)\vert\leq C_{\alpha_{1}\dots \alpha_{N}, \beta} \brkt{1+\abs{\xi_{1}}+\ldots \abs{\xi_{N}}}^{m-\varrho\sum_{j=1}^{N}\abs{\alpha_j}+ \delta |\beta|},
\end{equation*}
for some $m \in \R$, $\varrho, \delta\in[0,1]$ and all multi-indices $\alpha_1,\dots, \alpha_N , \beta$ in $\Z_{+}^n.$  However in this paper we shall also investigate multilinear FIOs that are rough (i.e. $L^p$) in the spatial variable $x$. The phase functions $\varphi_{j} (x,\xi_{j})$ in the definition of the multilinear FIO  are assumed to be $C^{\infty}(\R^n \times \R^n\setminus\{0\})$ and homogeneous of degree $1$ in their frequency variables. Furthermore, we require that the phase functions verify the {\it{strong non-degeneracy conditions}} $|\det\,\d^{2}_{x, \xi} \varphi_{j}(x,\xi)| \geq c_{j} >0$ for $j=1,\dots, N.$\\

Bilinear FIOs of the type above already appear in applications, for example in the study of non-linear wave equations in the work of D. Foschi and S. Klainerman \cite{FoscKlain}, which serves as our second motivation to consider these specific types of operators.

 When the phase functions $\varphi_{j}(x, \xi_{j})=\langle x,\xi_{j}\rangle$, the above operators are called {\it{multilinear pseudodifferential operators}}. There exists a great amount of literature concerning these operators, in particular for those that can be fit within the realm of multilinear Calder\'on-Zygmund theory. This case was thoroughly investigated and optimal results were obtained in the seminal work of L. Grafakos and R. Torres \cite{GrafTor}. But, for more general H\"ormander classes of bilinear operators that fall outside the scope of Calder\'on-Zygmund theory, there has been comparatively little amount of activity. To remedy this situation, attempts were made in \cite{MRS} and also in a work by A. Benyi, F. Bernicot, D. Maldonado, V. Naibo and R. Torres \cite{BMNT, BBMNT}, where the results in \cite{MRS} were improved an extended in many directions.

For the multilinear Fourier integral operators, the situation is rather different in that their study started recently in the paper of  L.~Grafakos and M.~Peloso \cite{GrafPelo}. In that paper the authors considered operators with phase functions that are  more general than those that will be considered here, but with more restrictive conditions on the support of the amplitudes involved and on their order and type. However, Grafakos and Peloso \cite{GrafPelo} also obtained results concerning boundedness of bilinear Fourier integral operators of the exact same type that are considered in this paper, and one of our goals here is to provide further extensions in that direction.

To summarize, the aims of this paper are to extend the existing boundedness results for rough linear $\Psi$DOs (e.g. those in \cite{KenStaub}) and FIOs (e.g. \cite{DosStaub}) and as a bi-product, improve on the results in \cite{MRS} and  \cite{BBMNT} on boundedness of H\"ormander class bilinear pseudodifferential operators, and also on those in \cite{GrafPelo} concerning bilinear FIOs that coincide with the ones considered here. Moreover we will also establish sharp results concerning boundedness of certain classes of Fourier integral operators with product type symbols (see Definition \ref{definition multilinear product type amplitudes}) and thereby generalize a result in \cite{Bernicot} to the setting of multilinear FIOs. We would also like to mention that several results here concerning multilinear operators are also valid without any smoothness assumptions in the spatial variable of the amplitudes .

\section{Classes of linear amplitudes and non-degenerate phase functions}\label{linear amplitudes}
In this section we define the classes of linear amplitudes with both smooth and rough spatial behaviour and also the class of phase functions that appear in the definition of operators treated here. In the sequel we use the notation $\langle \xi\rangle$ for $(1+|\xi|^2)^{1/2}.$ The following classical definition of amplitudes/symbols is due to H\"ormander \cite{Hor}.
\begin{definition}\label{definition of hormander amplitudes}
Let $m\in \R$, $0\leq \varrho,\delta\leq 1.$ A function $a(x,\xi)\in C^{\infty}(\R^{n} \times\R^{n})$ belongs to the class $S^{m}_{\varrho,\delta}$, if for any multi-indices $\alpha,\beta$
   it satisfies
   \begin{align*}
      \sup_{\xi \in \R^n} \langle \xi\rangle ^{-m+\varrho\vert \alpha\vert- \delta |\beta|}
      |\partial_{\xi}^{\alpha}\partial_{x}^{\beta}a(x,\xi)|< +\infty.
   \end{align*}
\end{definition}
We shall also deal with the class $L^{p}S^m_{\varrho}$ of rough amplitudes/symbols introduced by Michalowski, Rule and Staubach in \cite{MRS}, which is an extension of that introduced by Kenig and Staubach in \cite{KenStaub}.
\begin{definition} \label{LpSmrho definition}
Let $1 \leq p \leq \infty$, $m \in \R$ and $0\leq \varrho \leq 1$. A function $a(x,\xi)$, $x,\xi\in \R^n$ belongs to the class $L^p S^m_{\varrho},$ if $a(x,\xi)$ is measurable in $x\in \R^n ,$ $a(x,\xi)\in C^\infty(\R^n_\xi)$ a.e. $x\in \R^n$, and for each multi-index $\alpha$, there exists a constant $C_\alpha$ such that
\begin{equation*}
\|\partial_\xi^\alpha a(\cdot,\xi)\|_{L^p (\R^n)} \leq C_\alpha \langle \xi\rangle^{m-\varrho |
\alpha|}.
\end{equation*}
Here we also define the associated seminorms
\[
    \abs{a}_{p,m,s}=\sum_{\abs{\alpha}\leq s} \sup_{\xi\in \R^n}  \langle \xi\rangle^{\varrho |\alpha|-m} \norm{\partial_\xi^\alpha a(\cdot,\xi)}_{L^p (\R^n)}.
\]
\end{definition}
\begin{example} If $b\in L^p$ and ${a}(x,\xi)\in L^\infty S^m_\varrho$ then $b(x){a}(x,\xi)\in L^pS^m_\varrho$. In particular, the same holds for ${a}(x,\xi)\in S^{m}_{\varrho,\delta},$ with any $0\leq \varrho, \delta\leq 1.$
\end{example}

\begin{example}
 Take $\psi \in C_0^\infty(\R)$ with support in $[-1,1],$ and $h$ be an unbounded function in the Zygmund class $L_{\rm exp}[-1,1]$ (see \cite[Chp. 4]{BenSharp}). Then $e^{i\xi h(x)}\psi(x)\in L^pS^m_\varrho$. In particular $e ^{i\xi \log |x|}\psi(x)$ belongs to $L^pS^0_0$ for any $p<\infty$. Observe that in this case, for every $x\neq 0$,  $a(x,\xi)\in C^\infty$ and $\Vert \partial^\alpha_\xi a(\cdot,\xi)\Vert_{L^p}<\infty$ for all $p\neq \infty$, but for any $\alpha>0$, $\Vert\partial^\alpha_\xi a(\cdot,\xi)\Vert_{L^\infty}=+\infty$.

More generally, if $h, \psi$ are as above and $\sigma$ is a real valued function in $S^{m}_{\varrho,0}(\R)$  for $m\leq 0$ then $e^{ih(x)\,\sigma(\xi)}\psi(x)$ is in the class $L^pS^m_\varrho$ for any $p<\infty$.
\end{example}

We also have the following simple lemma concerning the products of rough amplitudes which follows directly from Leibniz's rule and H\"older's inequality.
\begin{lem}\label{lem:calculus_lemma} If $a\in L^pS^{m_1}_\varrho$ and $b\in L^q S^{m_2}_\varrho$ then $a\cdot b\in L^r S^{m_1+m_2}_{\varrho}$ where $\frac1 r =\frac1 p+\frac1 q$, $1\leq p,q\leq \infty$.
Moreover, if $\eta(\xi)\in C^{\infty}_{0}$ and $a_\varepsilon(x,\xi)=a(x,\xi)\eta(\varepsilon \xi)$ and $\varepsilon \in[0,1),$ then one has
\[
    \sup_{0<\varepsilon\leq 1}\sup_{\xi\in \R^n} \langle \xi\rangle^{-m+\varrho |
\alpha|} \norm{\d^\alpha_\xi a_\varepsilon (\cdot,\xi)}_{L^p}\leq c_{\eta,\abs{\alpha},\varrho}\abs{a}_{p,m,\abs{\alpha}}.
\]
\end{lem}

We also need to describe the class of phase functions that we will use in our investigation. To this end, the class $\Phi^{k}$  defined below, will play a significant role in our investigations.
\begin{definition} \label{definition Phik phases}
A real valued function $\varphi(x,\xi)$ belongs to the class $\Phi^{k}$, if $\varphi (x,\xi)\in C^{\infty}(\R^n \times\R^n \setminus\{0\})$, is positively homogeneous of degree $1$ in the frequency variable $\xi$, and satisfies the following condition:

For any pair of multi-indices $\alpha$ and $\beta$, satisfying $|\alpha|+|\beta|\geq k$, there exists a positive constant $C_{\alpha, \beta}$ such that
   \begin{equation}\label{C_alpha}
      \sup_{(x,\,\xi) \in \R^n \times\R^n \setminus\{0\}}  |\xi| ^{-1+\vert \alpha\vert}\vert \partial_{\xi}^{\alpha}\partial_{x}^{\beta}\varphi(x,\xi)\vert
      \leq C_{\alpha , \beta}.
   \end{equation}
\end{definition}

In connection to the problem of local boundedness of Fourier integral operators, one considers phase functions $\varphi(x,\xi)$ that are positively homogeneous of degree $1$ in the frequency variable $\xi$ for which $\det [\partial^{2}_{x_{j}\xi_{k}} \varphi(x,\xi)]\neq 0.$ The latter is referred to as the {\it{non-degeneracy condition}}. However, for the purpose of proving global regularity results, we require a stronger condition than the aforementioned weak non-degeneracy condition.

\begin{definition}\label{strong non-degeneracy}
 A real valued phase $\varphi\in C^{2}(\R^n \times\R^n \setminus\{0\})$ satisfies the {\it{strong non-degeneracy condition}} or the {\it{SND condition}} for short, if there exists a constant $c>0$ such that
\begin{equation}\label{lower_bound on mixed hessian}
\Big|\det \frac{\partial^{2}\varphi(x,\xi)}{\partial x_j \partial \xi_k}\Big| \geq c, \quad \text{for all}\quad (x,\,\xi) \in \R^n \times\R^n \setminus\{0\}
\end{equation}

\end{definition}

\begin{example} A phase function intimately connected to the study of the wave operator, namely $\varphi(x,\xi)= |\xi|+ \langle x, \xi\rangle$, satisfies the SND condition and belongs to the class $\Phi^2$.
\end{example}

As is common practice, we will denote constants which can be determined by known parameters in a given situation, but whose
value is not crucial to the problem at hand, by $C$. Such parameters in this paper would be, for example, $m$, $\varrho$, $p$, $n$ and the constants appearing in the definitions of various amplitude classes. The value of $C$ may differ from line to line, but in each instance could be estimated if necessary. We also write sometimes $a\lesssim b$ as shorthand for $a\leq Cb$.

\section{Tools in proving boundedness of rough linear FIOs}

Here we collect the main tools in proving our boundedness results for linear FIOs.

The following decomposition due to Seeger, Sogge and Stein \cite{SSS} is by now classical. One starts by taking a Littlewood-Paley partition of unity
\begin{equation}\label{eq:LittlewoodPaley}
    \Psi_0(\xi) +\sum_{j=1}^{\infty}\Psi_{j}(\xi)=1,
\end{equation}
where supp $\Psi_0\subset \{\xi;\,  \vert \xi \vert \leq 2 \}$, supp
$\Psi\subset \{\xi;\, \frac{1}{2} \leq \vert \xi \vert \leq 2 \}$
and $\Psi_{j}(\xi) =\Psi(2^{-j}\xi)$.

To get useful estimates for the amplitude and the phase function, one
imposes a second decomposition on the former Littlewood-Paley
partition of unity in such a way that each dyadic shell $2^{j-1}\leq
\vert \xi\vert\leq 2^{j+1}$ is partitioned into truncated cones of
thickness roughly $2^{\frac{j}{2}}$. Roughly $2^{\frac{(n-1)j}{2}}$
such elements are needed to cover the shell $2^{j-1}\leq \vert
\xi\vert\leq 2^{j+1}$. For each $j$ we fix a collection of unit
vectors $\{\xi^{\nu}_{j}\}_{\nu}$ that satisfy,
\begin{enumerate}
\item $\vert \xi^{\nu}_{j}-\xi^{\nu'}_{j}\vert\geq
2^{\frac{-j}{2}},$ if $\nu\neq \nu '$.
\item If $\xi\in\mathbb{S}^{n-1}$, then there exists a $
\xi^{\nu}_{j}$ so that $\vert \xi -\xi^{\nu}_{j}\vert
<2^{\frac{-j}{2}}$.
\end{enumerate}
Let $\Gamma^{\nu}_{j}$ denote the cone in the $\xi$ space whose
 central direction is $\xi^{\nu}_{j}$, i.e.
\begin{equation*}
    \Gamma^{\nu}_{j}=\set{ \xi;\,  \abs{ \frac{\xi}{\vert\xi\vert}-\xi^{\nu}_{j}}\leq 2\cdot 2^{\frac{-j}{2}}}.
\end{equation*}
One can construct an associated partition of unity given by
functions $\chi^{\nu}_{j}$, each homogeneous of degree $0$ in $\xi$
and supported in $\Gamma^{\nu}_{j}$ with,
\begin{equation*}
\sum_{\nu}\chi^{\nu}_{j}(\xi)=1,\,\,\, \text{ for all}\,\,\, \xi \neq
0\,\,\, \text{and all}\,\,\, j,
\end{equation*}
and
\begin{equation}
\label{chiestimate 1} \vert \partial
^{\alpha}_{\xi}\chi^{\nu}_{j}(\xi)\vert\leq C_{\alpha}
2^{\frac{\vert \alpha\vert j}{2}}\vert \xi\vert ^{-\vert \alpha
\vert},
\end{equation}
with the improvement
\begin{equation}
\label{chiestimate 2} \vert \partial
^{N}_{\xi_{1}}\chi^{\nu}_{j}(\xi)\vert\leq C_{N}
\vert \xi\vert ^{-N}, \quad  \text{for} \quad N\geq 1,
 \end{equation}
if one chooses the axis in $\xi$ space such that $\xi_1$ is in the direction of $\xi^{\nu}_{j}$ and $\xi'=(\xi _2 , \dots , \xi_{n})$ is perpendicular to $\xi^{\nu}_{j}$. Using $\Psi_{j}$'s and $\chi^{\nu}_{j}$'s, we can construct a
Littlewood-Paley partition of unity
\begin{equation*}
\Psi_{0}(\xi)+ \sum_{j=1}^{\infty}\sum_{\nu}\chi^{\nu}_{j}(\xi)\,
\Psi_{j}(\xi)=1.
\end{equation*}
Now to any of the classes of amplitudes and phases defined in Section \ref{linear amplitudes} one associates a Fourier integral operator given by
\begin{equation}\label{definition basic linear FIO}
  T_{a}f(x)= \frac{1}{(2\pi)^{n}}\int_{\R^{n}}e^{i
\varphi(x,\xi)}a(x,\xi)\hat{f}(\xi)
\, \dd\xi.
\end{equation}
Using the Littlewood-Paley decomposition above, we decompose this operator as
\begin{multline}
\label{Tdecomp}
T_0 f(x) + \sum_{j=1}^{\infty}\sum_{\nu}T^{\nu}_{j} f(x)=
\frac{1}{(2\pi)^{n}}\int_{\R^{n}}e^{i
\varphi(x,\xi)}a(x,\xi) \Psi_0(\xi)\hat{f}(\xi)
\, \dd\xi\\ +\frac{1}{(2\pi)^{n}}
\sum_{j=1}^{\infty}\sum_{\nu}\int_{\R^{n}}e^{i
\varphi(x,\xi)+i\langle x,\xi\rangle}a(x,\xi)\chi^{\nu}_{j}(\xi)
\Psi_{j}(\xi)\hat{f}(\xi)\, \dd\xi.
\end{multline}
We refer to $T_0$ as the {\it{low frequency part}}, and $T^{\nu}_{j}$ as the {\it{high frequency part}} of the FIO $T_{a}.$\\

Now, one introduces the phase function $\Phi(x,\xi)=
\varphi(x,\xi)-\langle (\nabla_{\xi}\varphi)(x,\xi^{\nu}_{j}),
\xi\rangle$ and the amplitude
\begin{equation}\label{amplitude}
A^{\nu}_{j}(x,\xi)= e^{i \Phi(x,\xi)}a(x,\xi)\chi^{\nu}_{j}(\xi)\,
\Psi_{j}(\xi).
\end{equation}

\noindent It can be verified (see e.g. \cite[p. 407]{Stein}) that the phase
$\Phi(x,\xi)$ satisfies 
\begin{eqnarray}
    \label{phaseestim1} \vert (\frac{\partial}{\partial \xi_1})^{N}
\Phi(x,\xi)\vert\leq C_{N} 2^{-Nj},\quad \text{and}\\
\label{phaseestim2} \vert (\nabla_{\xi'})^{N} \Phi(x,\xi)\vert\leq
C_{N} 2^{\frac{-Nj}{2}},
\end{eqnarray}
for $N\geq 2$ on the support of $A^{\nu}_{j}(x,\xi)$.

Using these, we can rewrite $T^{\nu}_{j}$ as a FIO with
a linear phase function,
\begin{equation}
\label{definition rewrittentnuj} T^{\nu}_{j} f(x)
=\frac{1}{(2\pi)^{n}}\int_{\R^{n}}A^{\nu}_{j}(x,\xi) e^{i
\langle (\nabla_{\xi}\varphi)(x,\xi^{\nu}_{j}),\, \xi\rangle} \hat{f}(\xi)\, \dd\xi.
\end{equation}

In this paper we will only deal with classes $\Phi^1,$ and more importantly $\Phi^2 ,$ of phase functions. In the case of class $\Phi^{2},$ we have only required control of those frequency derivatives of the phase function which are greater or equal to $2$. This restriction is motivated by the simple model case phase function $\varphi(x,\xi)=|\xi|+ \langle x,\xi\rangle$ for which the first order $\xi$-derivatives of the phase are not bounded but all the derivatives of order equal or higher than 2, decay away from the origin and so $\varphi(x,\xi)\in\Phi^2$. However in order to handle the boundedness of the low frequency parts of FIOs, one also needs to control the first order $\xi$ derivatives of the phase.

The following phase reduction will reduce the phase to a linear term plus a phase for which the first order frequency derivatives are bounded. The proof can be found in \cite[Lemma 1.2.3]{DosStaub} for amplitudes in $L^\infty S^m_\varrho$, but the same argument also holds for amplitudes in any $L^p S^m_\varrho$.

\begin{lem} 
\label{lem:change}
Any FIO $T_{a}$ of the type \eqref{definition basic linear FIO} with amplitude $a(x,\xi)\in L^{p}S^{m}_{\varrho}$ and phase function $\varphi(x,\xi)\in \Phi^2$, can be written as a finite sum of Fourier integral operators of the form
\begin{equation}\label{reduced rep of Fourier integral operator}
  \frac{1}{(2\pi)^{n}} \int a(x,\xi)\, e^{i\psi(x,\xi)+i\langle \nabla_{\xi}\varphi(x,\zeta),\xi\rangle}\, \widehat{f}(\xi) \, \dd\xi
\end{equation}
where $\zeta$ is a point on the unit sphere $\mathbb{S}^{n-1}$, $\psi(x,\xi)\in \Phi^1$ and $a(x,\xi) \in L^{p} S^{m}_{\varrho}$ is localized in the $\xi$ variable around the point $\zeta$.

\end{lem}

We will also need a uniform non-stationary phase estimate that yields a uniform bound for certain oscillatory integrals that arise as kernels of certain  operators. To this end, we have:
\begin{lem}\label{lem:technic} Let $\mathcal{K}\subset \R^n$ be a compact set, $U\supset \mathcal{K}$ an open set and $k$ a nonnegative integer. Let $\phi$ be a real valued function in $C^{\infty}(U)$ such that $\abs{\nabla \phi}>0$ and 
\[
    \abs{\d^\alpha \phi}\lesssim \abs{\nabla \phi},\qquad  \abs{\d^\alpha \brkt{\abs{\nabla \phi}^2}}\lesssim \abs{\nabla \phi}^2, \quad \text{for all multi-indices $\alpha$ with $\abs{\alpha}\geq 1$}
\]
Then, for any $F\in C^\infty_0(\mathcal{K})$, any integer $k\geq 0$ and any $\lambda>0$,
\begin{equation*}
	\lambda ^k \abs{\int_{\R ^n} F(\xi)\, e^{i\lambda \phi(\xi)}\, \dd \xi}\leq C_{k,n,\mathcal{K}} \sum_{\abs{\alpha}\leq k} \int_\mathcal{K} \abs{\d^{\alpha} F(\xi)} \abs{\nabla \phi(\xi)}^{-k}\, \dd \xi.
\end{equation*}
\begin{proof}
    Let $\Psi=\abs{\nabla \phi}^2$.  We observe first that for any multi-index $\alpha$ with $\abs{\alpha}\geq 0$, $\abs{\partial^\alpha\brkt{1/\Psi}}\lesssim \brkt{1/\Psi}$. The assertion is trivial for $\abs{\alpha}=0$. Let $\abs{\alpha}\geq 1$ and suppose that $\abs{\d^\gamma (1/\Psi)}\lesssim 1/\Psi$ for any multi-index $\gamma$ with $\abs{\gamma}<\abs{\alpha}$. Leibniz rule yields
    \[
        \partial^\alpha\brkt{1/\Psi} \Psi=-\sum_{\beta<\alpha}\binom{\alpha}{\beta}  \partial^\beta\brkt{1/\Psi}  \partial^{\alpha-\beta}\brkt{\Psi},
    \]
   from which, by the induction hypothesis and the assumption on $\Psi$, the claim follows.

    Let us define $A_0=F$ and
    \[
        A_{k}^{j_1,\ldots,j_k}=\d_{j_l}\brkt{A^{j_1,\ldots,j_{k-1}}_{k-1} {\d_{j_l} \phi}/{\Psi}},
    \]
    for $k\geq 1$, $j_l\in \set{1,\ldots,n}$ for $l\in\{0,\ldots,k\}$.
    Observe that, for any multi-index $\alpha$, $\abs{\alpha}\geq 0$,
    \[
        \begin{split}
        \d^\alpha \brkt{A_{k}^{j_1,\ldots,j_k}}
        &=\sum \binom{\alpha}{\beta}\binom{\beta}{\gamma} \Bigg( \d^\beta \d_{j_k}A_{k-1}^{j_1,\ldots,j_{k-1}}\, \d^\gamma \d_{j_k}\phi \, \d^{\alpha-\beta-\gamma}\brkt{1/\Psi}\\
        &+\d^\beta A_{k-1}^{j_1,\ldots,j_{k-1}}\, \d^\gamma \d_{j_k,j_k}^2\phi \,\d^{\alpha-\beta-\gamma}\brkt{1/\Psi}\\
        &+\d^\beta \d_{j_k} A_{k-1}^{j_1,\ldots,j_{k-1}}\, \d^\gamma \d_{j_k}\phi \, \d^{\alpha-\beta-\gamma}\d_{j_k}\brkt{1/\Psi}\Bigg).
        \end{split}
    \]
    Proceeding by induction, one can see that for $k\geq 1$ and for any multi-index $\alpha$ with $\abs{\alpha}\geq 0$,
    $A_k^{j_1,\ldots,j_k}\in C^{\infty}_{0}(\mathcal{K})$ and
    \begin{equation}\label{eq:techn}
        \abs{\partial^\alpha A_k^{j_1,\ldots,j_k}}\lesssim \sum_{\abs{\beta}\leq \abs{\alpha}+k} \abs{\d^\beta F} \Psi^{-k/2}.
    \end{equation}
    Since $1=\sum_{j=1}^n \frac{\partial_{j}\phi}{\Psi} \partial_j \phi$, and $i \lambda \partial_j\phi e^{i\lambda \phi}= \d_{j}\brkt{e^{i\lambda \phi}}$, integration by parts yields
    \[
        (-i\lambda)^k\int_{\R ^n} F(\xi) e^{i\lambda \phi(\xi)}\, \dd \xi=\sum_{j_1,\ldots,j_k=1}^n \int_{\mathcal{K}} A^{j_1,\ldots,j_k}_k e^{i\lambda \phi(\xi)}\, \dd \xi.
    \]
    Then the result follows by taking absolute values of both sides and using \eqref{eq:techn} for $\abs{\alpha}=0$.
\end{proof}
\end{lem}

\section{Global $L^q -L^r$ boundedness of rough linear Fourier integral and Pseudodifferential operators}\label{linear FIOs}
In this section we shall state and prove a boundedness results for rough $\Psi$DOs
and FIOs (with smooth strongly non-degenerate phase
functions), extending the results in \cite{DosStaub,KenStaub}.
\subsection{Boundedness of FIOs}\label{subsec low freq}
 First we deal with the boundedness of FIOs by doing a separate analysis of the low and high frequency parts of the operator. Using the decomposition in \eqref{Tdecomp}, we shall first establish the boundedness of the low frequency portion of the Fourier integral operator given by
\[T_0 f(x)=\frac{1}{(2\pi)^{n}}\int_{\R^{n}}e^{i
\varphi(x,\xi)}a(x,\xi) \Psi_0(\xi)\hat{f}(\xi)
\, \dd\xi,\]
where $\Psi_0 \in C_{0}^{\infty}$ and is supported near the origin. Clearly, instead of studying $T_0$, we can consider an FIO $T_{a}$ whose amplitude $a(x,\xi)$ is compactly supported in the  frequency variable $\xi.$ In what follows, we shall adopt this and drop the reference to $T_0$.
But before, we proceed with the investigation of the $L^q-L^r$ boundedness, we will need the following lemma, whose proof is a straightforward application of \cite[Lemma $1.2.10$]{DosStaub}

\begin{lem} \label{lem:fuijiwara}
 Let $\eta(\xi)$ be a $C_{0}^\infty$ function and set
\[
	K(x,z)=\int_{\R^n} \eta(\xi) e^{i(\psi(x,\xi)+\pr{z}{\xi})}\, \dd \xi,
\]
where $\psi(x,\xi)\in \Phi^1 .$
Then for any $\alpha\in (0,1)$, there exists a positive constant $c$ such that
\[
	\abs{K(x,z)}\leq c (1+\abs{z})^{-n-\alpha}.
\]
\end{lem}

\begin{remark}
In what follows the norms of the operators involved will depend on various parameters and a finite number of seminorms of the corresponding amplitudes (as in Definition $\ref{LpSmrho definition}$). Therefore, we refrain from emphasizing this dependence in the statement of the theorems.
\end{remark}

\begin{thm}
\label{Linearlow}
Suppose that $0<r\leq \infty$, $1\leq p,q\leq \infty$ satisfy the relation $\frac 1 r = \frac 1 q + \frac 1 p$. Assume that $\varphi\in \Phi^2$ satisfies the SND condition and let $a\in L^pS^m_\varrho$ with $m\leq 0$, $0\leq \varrho \leq 1$, such that $\supp_{\xi}a (x,\xi)$ is compact. Then the FIO $T_a$ defined as in \eqref{definition basic linear FIO} is bounded from $L^q$ to $L^r$.
\begin{proof}
Consider a closed cube $Q$ of side-length $L$ such that $\supp_{\xi} a(x,\xi) \subset \text{Int}(Q).$ We extend $a(x,\cdot)|_{Q}$ periodically with period $L$ into $\widetilde{a}(x,\xi)\in C^{\infty}(\R^{n}_\xi).$ Let $\eta\in C^{\infty}_{0}$ with $\supp \eta \subset Q$ and $\eta=1$ on $\xi$-support of $a(x,\xi)$, so we have $a(x,\xi)=\widetilde{a}(x,\xi) \eta(\xi)$. Now if we expand $\widetilde{a}(x,\xi) $ in a Fourier series, setting $f_k(x)=f(x-\frac{2\pi k}{L})$ for any $k\in \Z^n ,$ we can write
\begin{equation}\label{eq:Fourier_Serie}
	T_a f(x)=\sum_{k\in \Z^n} a_k(x) T_\eta (f_k)(x),
\end{equation}
where
\[
	a_k(x)=\frac{1}{L^n}\int_{\R^n} a(x,\xi) e^{-i \frac{2\pi}{L}\pr{k}{\xi}}\, \dd \xi,
\]
and $T_\eta(f_k)(x)=(2\pi)^{-n}\int \eta(\xi) e^{i \varphi(x,\xi)}\widehat{f_k}(\xi)\, \dd \xi.$ Let us assume for a moment that $T_\eta$ is a bounded operator on $L^q$. Take  $l=1,\ldots, n$ such that $\abs{k_l}\neq 0$. Integration by parts yields
\[
	a_k(x)= \frac{c_{n,N}}{|k_{l}|^N}\int_{\R^n} \partial^N_{\xi_l} a(x,\xi) e^{-i  \frac{2\pi}{L}\pr{k}{\xi}}\, \dd \xi.
\]
Observe also that, by the hypothesis on the amplitude and Lemma \ref{lem:calculus_lemma}
\begin{equation*}
	\max_{s=0,\ldots, N}\int_{\R^n} \norm{\partial^s_{\xi_{l}} a(\cdot,\xi)}_{L^p}\, \dd\xi\leq c_{n, N,\varrho}\abs{a}_{p,m,N},
\end{equation*}
for $N=[\max(n,n/r)]+1$. Thus
\begin{equation}\label{estimak}
	\norm{a_k}_{L^p}\leq c_{n, N,\varrho}\abs{a}_{p,m,N} (1+\abs{k})^{-N}.
\end{equation}
Let us first assume that $r\geq 1$. Then the Minkowski and H\"older inequalities yield

\begin{equation}\label{eq:bound}
\begin{split}
	\norm{T_a f}_{L^r}\leq  \sum_{k\in \Z^n}  \norm{a_k T_{\eta}(f_k)}_{L^r}
\leq  \sum_{k\in \Z^n} \norm{a_k}_{L^p} \norm{T_{\eta}(f_k)}_{L^{q}}.
\end{split}
\end{equation}
On the other hand, since we have assumed that $T_\eta$ is bounded on $L^q$ and the translations are isometries on $L^q$, we have that $\norm{T_\eta (f_k)}_{L^q}\leq  c_{\eta,\varphi} \norm{f}_{L^q}$. Therefore \eqref{estimak} yield
\[
	\norm{T_a f}_{L^r}\lesssim \abs{a}_{p,m,N} \sum_{k\in \Z^n} (1+|k|)^{-N} \norm{f}_{L^q}\thickapprox \norm{f}_{L^q}.
\]

Assume now that $0<r<1$. Using \eqref{eq:Fourier_Serie} and H\"older's inequality we have
\[
	\int \abs{T_a f(x)}^r \, \dd x
	\leq \sum_{k\in \Z^n}  \int \abs{T_\eta (f_k)(x)}^r\abs{a_k(x)}^r\, \dd x\leq 	\sum_{k\in \Z^n}\norm{a_k}_{L^p}^r \norm{T_{\eta}(f_k)}_{L^{q}}^r.
\]
The boundedness assumption on $T_\eta$ and \eqref{estimak} yields
\[
  \int \abs{T_a f(x)}^r \, \dd x 	\lesssim \abs{a}_{p,m,N}^r \sum_{k\in \Z^n} (1+|k|)^{-Nr} \norm{f}_{L^q}^r\thickapprox \norm{f}_{L^q}^r.
\]

In order to finish the proof we have to show that $T_\eta$ defines a bounded operator on $L^q$, for $1\leq q\leq \infty$. By Lemma \ref{lem:change} we can assume without loss of generality that
\[
	\varphi(x,\xi)=\psi(x,\xi) + \langle \textbf{t}(x), \xi\rangle,
\]
with a smooth map $\textbf{t}: \R^{n}\to \R^n$. 

For $f\in \S$ one has
\begin{equation}\label{eq:kernel_def}
	T_\eta(f)(x)=\frac{1}{(2\pi)^n}\int \eta(\xi) e^{i \pr{\xi}{\mathbf{t}(x)}} e^{i \psi(x,\xi)}\widehat{f}(\xi)\, \dd \xi=\int  K(x,\mathbf{t}(x)-y) f(y)\, \dd y,
\end{equation}
with
\begin{equation}\label{low frequency kernel estim}
	K(x,z)=\frac{1}{(2\pi)^n}\int \eta(\xi) e^{i \pr{\xi}{z}} e^{i \psi(x,\xi)}\, \dd \xi.
\end{equation}
Now, it follows from Lemma \ref{lem:fuijiwara} that for any  $\alpha\in (0,1)$, there exists a constant $c$ such that
\[
	\abs{K(x,z)}\leq c (1+\abs{z})^{-n-\alpha},
\]
and therefore $\sup_{x} \int | K(x,\mathbf{t}(x)-y) |\,\dd y <\infty.$ This yields the boundedness of the operator $T_\eta$ on $L^{\infty}.$  Moreover using the change of variables $z=\mathbf{t}(x)$, the SND condition yields that $\abs{{\rm det}\, D \textbf{t}(x)}\geq c>0$. Therefore if we denote the Jacobian of the change of variables by $J(z)$, J. Schwartz's global inverse function theorem \cite[Theorem 1.22]{Schwartz} implies that $\mathbf{t}$ is a global diffeomorphism on $\R^n$ and  $\abs{\det\, J(z)}\leq 1/c$. Thus
\[
    \begin{split}
	\sup_{y}\int  | K(x,\mathbf{t}(x)-y) |\, \dd x &=\sup_{y}\int |K(\mathbf{t}^{-1}(z),z-y) ||\det\,J(z)|\, \dd z  \\
        &\leq \frac{1}{c} \sup_{y}\int (1+\abs{z-y})^{-n-\alpha}\, \dd z < \infty,
    \end{split}
\]
where we have also used \eqref{low frequency kernel estim}. Therefore Schur's lemma yields that $T_\eta$ is bounded on $L^q$ for all $1\leq q\leq \infty$ and this ends the proof of the theorem.
\end{proof}
\end{thm}

Now we proceed to the proof of the general case.
\begin{thm}\label{thm:Linear}
 Suppose that $0<r\leq \infty$, $1\leq p,q \leq \infty$, satisfy the relation $\frac 1 r = \frac 1 q + \frac 1 p$. 
Assume that $\varphi\in \Phi^2$ satisfies the SND condition and let $a\in L^pS^m_\varrho$ with $0\leq \varrho \leq 1$ and 
\begin{equation}\label{eq:m_1}
	m < -\frac{(n-1)}{2}\brkt{\frac{1}{s}+\frac{1}{\min(p,s')}} + \frac{n(\varrho-1)}{s},
\end{equation}
where $s= \min(2,p,q)$ and $\frac{1}{s}+\frac{1}{s'}=1$. Then  the FIO $T_a$ defined as in \eqref{definition basic linear FIO} is bounded from $L^q$ to $L^r$.
\end{thm}

\begin{proof}
We shall assume that $q<\infty$. The case $q=\infty$ is proved with minor modifications in the argument, so we omit the details. We would like to prove that
\[
\|T_a f\|_{L^r (\R^n)} \leq C\|f\|_{L^q (\R^n)}, \quad \text{for all}\quad f\in \S.
\]
To achieve this, we decompose $T_a$ as in \eqref{Tdecomp} in the form $T_0f+ \sum_{j=1}^{\infty}\sum_{\nu}T^{\nu}_{j} f(x)$. By Theorem \ref{Linearlow}, the first term $T_0$, satisfies the desired boundedness, so we confine ourselves to the analysis of the second term. Here we use the representation \eqref{definition rewrittentnuj} of the operators $T^{\nu}_{j}$ namely
\begin{equation*}
 T^{\nu}_{j} f(x)
=\frac{1}{(2\pi)^{n}}\int_{\R^{n}}A^{\nu}_{j}(x,\xi) e^{i
\langle (\nabla_{\xi}\varphi)(x,\xi^{\nu}_{j}),\, \xi\rangle} \hat{f}(\xi)\, \dd\xi=\int_{\R^{n}} K^{\nu}_{j} (x,(\nabla_{\xi}\varphi)(x,\xi^{\nu}_{j})-y) f(y) \dd y,
\end{equation*}
where

\begin{equation*}
K^{\nu}_{j} (x,z)
=\frac{1}{(2\pi)^{n}}\int_{\R^{n}}A^{\nu}_{j}(x,\xi) e^{i
\langle z,\, \xi\rangle}\, \dd\xi.
\end{equation*}
Let $L$ be the differential operator given by
\begin{equation*}
L=I-2^{2j}\frac{\partial ^{2}}{\partial \xi_1 ^{2}}-2^{j}
\Delta_{\xi'}.
\end{equation*}
Using the definition of $A^{\nu}_{j} (x,\xi)$ in \eqref{amplitude}, the assumption that $a\in L^{p}S^{m}_{\varrho}$ together with \eqref{chiestimate 1}, \eqref{chiestimate 2}, and the uniform estimates (in $x$) for $\Phi(x,\xi)$ in \eqref{phaseestim1} and \eqref{phaseestim2}, we can show that for any $\nu$ and any $\xi\in \sup_\xi {A^\nu_j}$
\begin{equation*}
\Vert L^N A^{\nu}_{j}(\cdot,\xi)\Vert_{L^{p}}\leq C_{N} 2^{j(m+ 2N(1-\varrho))}.
\end{equation*}
Let ${\mathbf t}_j^\nu(x)=(\nabla_{\xi}\varphi)(x,\xi^{\nu}_{j})$ and $\alpha\in (0,\infty)$. As before, the SND condition on the phase function yields that $|\det D {\mathbf t}_j^\nu(x)|\geq c>0.$ Setting
\begin{equation*}
g(y)= (2^{2j} y^2_1 + 2^{j} |y'|^2)^{\frac{\alpha}{2}},
\end{equation*}
we can split
\[
\begin{split}
\textbf{I}_1 + \textbf{I}_2&=\sum_{\nu}\left(\int_{g(y)\leq 2^{-j\varrho}} +\int_{g(y)> 2^{-j\varrho}}\right)\vert K_{j}^{\nu}(x,y)f({\mathbf t}_j^\nu(x)-y) \vert\, \dd y\\
&=\sum_{\nu}\int\vert K_{j}^{\nu}(x,y)f({\mathbf t}_j^\nu(x)-y)\vert\, \dd y.
\end{split}
\]
H\"older's inequality in $\nu$ and $y$ simultaneously and thereafter, since $1\leq s\leq 2$, the Hausdorff-Young inequality in the $y$ variable of the second integral yield
\begin{equation*}
\begin{split}
	\textbf{I}_1 &\leq \sqbrkt{\sum_{\nu}\int_{g(y)\leq 2^{-j\varrho}} \abs{f({\mathbf t}_j^\nu(x)-y)}^s \dd y}^{\frac{1}{s}}\sqbrkt{\sum_{\nu}\int\vert K_{j}^{\nu}(x,y)\vert^{s'} \dd y}^{\frac{1}{s'}}\\
&\lesssim  \sqbrkt{\sum_{\nu}\int_{g(y)\leq 2^{-j\varrho}} \abs{f({\mathbf t}_j^\nu(x)-y)}^s \dd y}^{\frac{1}{s}}\sqbrkt{\sum_{\nu}\brkt{\int\vert A_{j}^{\nu}(x,\xi)\vert^{s}\, \dd\xi}^{\frac{s'}{s}}}^{\frac{1}{s'}}.
\end{split}
\end{equation*}
If we now set $F_{j}^\nu(x,y)=f({\mathbf t}_j^\nu(x)-y),$ raise the expression in the estimate of $\textbf{I}_1$ to the $r$-th power and integrate in $x$, then H\"older's inequality yields that $\norm{\textbf{I}_1}_{L^r}$ is bounded by a constant times
\begin{equation}\label{I}
\sqbrkt{\int \brkt{\sum_{\nu}\int_{g(y)\leq 2^{-j\varrho}} \abs{F^{\nu}_j(x,y)}^s \dd y}^{\frac{q}{s}} \dd x}^{\frac{1}{q}}
\sqbrkt{ \int \brkt{\sum_{\nu}\brkt{\int\vert A_{j}^{\nu}(x,\xi)\vert^{s}\, \dd\xi}^{\frac{s'}{s}}}^{\frac{p}{s'}} \dd x}^{\frac{1}{p}}.
\end{equation}
We shall deal with the two terms in the right hand side of this estimate separately. To this end using the Minkowski integral inequality (simultaneously in $y$ and $\nu$), we can see that the first term is bounded by
\[
	\sqbrkt{\sum_{\nu}\int_{g(y)\leq 2^{-j\varrho}}  \brkt{\int \abs{F^{\nu}_j(x,y)}^q \dd x}^{\frac{s}{q}} \dd y}^{\frac{1}{s}}.
\]
Observe now that, letting ${\mathbf t}_j^\nu(x)=t$ and using $\abs{\det D\,{\mathbf t}_j^\nu(x)}\geq c>0$, we obtain
\begin{equation}\label{eq:u}
	\brkt{\int \abs{F^{\nu}_j(x,y)}^q \dd x}^{\frac{1}{q}}=\brkt{\int \abs{f(t-y)}^q \abs{\det D\,{\mathbf t}_j^\nu(x)}^{-1} \dd t}^{\frac{1}{q}}\leq c^{-\frac{1}{q}} \norm{f}_{L^q}.
\end{equation}
Thus, the first term on the right hand side of \eqref{I} is bounded by a constant multiple of
\begin{equation} \label{first term of I}
\begin{split}
	\sqbrkt{\sum_\nu \int_{g(y)\leq 2^{-j\varrho}} \dd y}^{\frac{1}{s}}  \norm{f}_{L^q}&\lesssim
	2^{j\frac{n-1}{2s}} 2^{-j\frac{n+1}{2s}}  \sqbrkt{\int_{\abs{y}\leq 2^{-j\frac{\varrho}{\alpha}}} \dd y}^{\frac{1}{s}} \norm{f}_{L^q}\\&\lesssim 2^{j\frac{n-1}{2s}} 2^{-j\frac{n+1}{2s}}  2^{-j\frac{n\varrho}{\alpha s}} \norm{f}_{L^q}.
\end{split}
\end{equation}
To analyse the second term we shall consider two separate cases, so assume first that $p\geq s'.$ Minkowski inequality yields that the second term in the right hand side of \eqref{I} is bounded by
\begin{equation*}
\begin{split}
\set{\sum_{\nu} \sqrbrak{\int \brkt{\int |A^{\nu}_{j} (x,\xi)|^{s}\dd \xi}^{\frac{p}{s}}\dd x}^{\frac{s'}{p}}}^{\frac{1}{s'}}&\leq  \set{\sum_{\nu} \sqrbrak{\int \brkt{\int |A^{\nu}_{j} (x,\xi)|^{p}\dd x}^{\frac{s}{p}}\dd \xi}^{\frac{s'}{s}}}^{\frac{1}{s'}}\\ &\lesssim 2^{jm} \brkt{\sum_{\nu} |\supp_{\xi} A^{\nu}_{j}|^{\frac{s'}{s}}}^{\frac{1}{s'}}\lesssim 2^{jm} 2^{j\frac{n+1}{2s}}2^{j\frac{n-1}{2s'}},
\end{split}
\end{equation*}
where we have used the fact that the measure of the $\xi-$support of $A^{\nu}_{j}$ is $O(2^{j\frac{n+1}{2}}).$ If $p<s'$, the second term on the right hand side of \eqref{I} is bounded by
\begin{equation*}
\begin{split}
\set{\sum_{\nu} {\int \brkt{\int |A^{\nu}_{j} (x,\xi)|^{s}\dd \xi}^{\frac{p}{s}}\dd x}}^{\frac{1}{p}}&\leq  \set{\sum_{\nu} \sqrbrak{\int \brkt{\int |A^{\nu}_{j} (x,\xi)|^{p}\dd x}^{\frac{s}{p}}\dd \xi}^{\frac{p}{s}}}^{\frac{1}{p}}\\ &\lesssim 2^{jm} \brkt{\sum_{\nu} |\supp_{\xi} A^{\nu}_{j}|^{\frac{p}{s}}}^{\frac{1}{p}}\lesssim 2^{jm} 2^{j\frac{n+1}{2s}}2^{j\frac{n-1}{2p}}.
\end{split}
\end{equation*}
Therefore, \eqref{first term of I} and the previous estimates yield
\[
	\norm{\textbf{I}_1}_{L^r} \lesssim 2^{j\brkt{m-\varrho \frac{n}{\alpha s}+\frac{n-1}{2}\brkt{\frac{1}{s}+\frac{1}{\min(p,s')}}}}\norm{f}_{L^q},
\]
where the constant hidden on the right hand side of this estimate does not depend on $\alpha$.

Define $h(y)=1+2^{2j} y^2_1 + 2^{j} |y'|^2$ and let $M>\frac{n}{2s}$. By H\"older's inequality,
\begin{multline}\label{eq:I2_bound}
	\norm{\textbf{I}_2}_{L^r}\leq\sqbrkt{\int\brkt{\sum_\nu\int_{g(y)>  2^{-j\varrho}}\abs{F^{\nu}_j(x,y)}^sh(y)^{-sM}\, \dd y}^{\frac{q}{s}}\, \dd x}^{\frac{1}{q}}\\
	\times \sqbrkt{\int \brkt{\sum_\nu\int\vert {K_{j}^{\nu}}(x,y)\,h(y)^{M}\vert^{s'}\, \dd y}^{\frac{p}{s'}}\, \dd x}^{\frac{1}{p}}.
\end{multline}
By Minkowski's integral inequality and \eqref{eq:u}, the first term of the right hand side is bounded by
a constant times \begin{equation}\label{eq:estimate1}
	\begin{split}
		\norm{f}_{L^q} \sqbrkt{\sum_\nu\int_{g(y)>  2^{-j\varrho}}h(y)^{-sM}\, \dd y}^{\frac{1}{s}} &\lesssim \norm{f}_{L^q}2^{j\frac{n-1}{2s}}2^{\frac{-j(n+1)}{2s}}\sqbrkt{\int_{|y|> 2^{-j\frac{\varrho}{\alpha}}}|y|^{-2sM}\, \dd y}^{\frac{1}{s}} \\
		&\lesssim \norm{f}_{L^q }2^{j\frac{n-1}{2s}} 2^{\frac{-j(n+1)}{2s}}  2^{j\frac{\varrho}{\alpha} (2M-\frac{n}{s})}.
	\end{split}
\end{equation}
In order to control the second term in \eqref{eq:I2_bound}, let us first assume that $M\in \Z_+$. In this case, Hausdorff-Young's inequality, Minkowski's integral inequality, and the same argument as in the analysis of $\textbf{I}_1$ yield
\begin{equation}\label{integer M estimates for I2}
	\begin{split}
	\set{\int \sqbrkt{\sum_\nu\int\vert {K_{j}^{\nu}}(x,y)\,h(y)^{M}\vert^{s'}\, \dd y}^{\frac{p}{s'}}\, \dd x}^{\frac{1}{p}} &\leq
	\set{\int \sqbrkt{\sum_\nu\brkt{\int\abs{L^M A_{j}^{\nu}(x,\xi)}^s\, \dd \xi}^{\frac{s'}{s}}}^{\frac{p}{s'}}\, \dd x}^{\frac{1}{p}}\\
	&\lesssim 2^{j(m+ 2M(1-\varrho))} 2^{j\frac{n+1}{2s}} 2^{j\frac{n-1}{2\min(s',p)}}.
	\end{split}
\end{equation}
The same estimate for non-integral values of $M$ is also valid by a standard argument, writing $M$ as $[M]+\{M\}$ where $[M]$ denotes the integer part of $M$ and $\{M\}$ its fractional part and using  H\"older's inequality with conjugate exponents $\frac{1}{\{M\}}$, $\frac{1}{1-\{M\}}$ and \eqref{integer M estimates for I2}.

Hence, for every $2M>\frac{n}{s}$, \eqref{eq:estimate1} and \eqref{integer M estimates for I2} yield
\[
	\norm{\textbf{I}_2}_{L^r} \lesssim 2^{j(m+2M(1-\varrho))}2^{j\frac{\varrho}{\alpha} (2M-\frac{n}{s})} 2^{j\frac{n-1}{2}\brkt{\frac{1}{s}+\frac{1}{\min(p,s')}}}\norm{f}_{L^q},
\]
with a constant independent of $\alpha$.
Now putting the estimates for $\textbf{I}_1$ and $\textbf{I}_2$ together and summing, yield that for any $\alpha>0$,
\[
    \norm{T_j f}_{L^r}\lesssim \brkt{2^{j(m+2M(1-\varrho)+\frac{n-1}{2}\brkt{\frac{1}{s}+\frac{1}{\min(p,s')}})}2^{j\frac{\varrho}{\alpha} (2M-\frac{n}{s})}+2^{j\brkt{m-\varrho \frac{n}{\alpha s}+\frac{n-1}{2}\brkt{\frac{1}{s}+\frac{1}{\min(p,s')}}}}}\norm{f}_{L^q}.
\]
Therefore, letting $\alpha$ tend to $\infty,$ we obtain
\[
    \norm{T_j f}_{L^r}\lesssim 2^{j\sqbrkt{m+2M(1-\varrho)+\frac{n-1}{2}\brkt{\frac{1}{s}+\frac{1}{\min(p,s')}}}}\norm{f}_{L^q}.
\]
Now if we let $R=\min(r,1),$ we obtain
\[
	\norm{\sum_{j=1}^{\infty}T_{j} f}_{L^r}^R\leq \sum_{j=1}^{\infty}\norm{T_{j} f}_{L^r}^R\lesssim
	\sum_{j=1}^\infty 2^{jR \sqbrkt{\frac{n-1}{2}\brkt{\frac{1}{s}+\frac{1}{\min(p,s')}}+m+2M(1-\varrho)}}\norm{f}_{L^q}^R\lesssim \norm{f}_{L^q}^R,
\]
provided  $m< -\frac{n-1}{2}\brkt{\frac{1}{s}+\frac{1}{\min(p,s')}}+2M(\varrho-1)$, for some $2M>\frac{n}{s}$, which exists by \eqref{eq:m_1}.
\end{proof}

In the case of $q=2\leq p,$ Theorem \ref{thm:Linear} can be improved, but before we proceed to that, we will need a lemma.

\begin{lem}\label{lem:H_M} Let $2\leq p\leq \infty$, $0\leq \varrho\leq 1$, $a \in L^p S^m_\varrho$ and $r=\frac{2p}{p+2}.$ For $f\in \S,$ a real number $M>n,$ and all multi-indices $\alpha,\beta$ with $\beta\leq \alpha$, set
\begin{equation}
\label{eq:H_M}
	H_M^{\alpha,\beta} f(x,\xi)= |\partial^{\beta}_{\xi} a(x,\xi)| {\int \brkt{1+2^j\abs{x-y}}^{-M} | \partial^{\alpha-\beta}_{\xi} a(y,\xi)| |f(y)|\, \dd y }.
\end{equation}
Then for every $f\in L^{r'}$
\[
	\norm{H_M^{\alpha,\beta}f(\cdot,\xi)}_{L^{r}}\leq C_{M}
	\abs{a}_{p,m,\abs{\alpha}}^2  2^{-jn} \left< \xi\right>^{2m-\varrho\abs{\alpha}}\norm{f}_{L^{r'}}.
\]
\begin{proof}
Since $\frac{1}{r}=\frac{1}{p}+\frac{1}{2}$, H\"older's and Minkowski's inequalities yield
\[
	\begin{split}
	\norm{H_M^{\alpha,\beta}f(\cdot,\xi)}_{L^{r}}&\leq \Vert{\partial^{\beta}_{\xi} a(\cdot,\xi)}\Vert_{L^p}
	 \norm{{\int  \brkt{1+2^j\abs{y}}^{-M} |\partial^{\alpha-\beta}_{\xi} a(\cdot-y,\xi) f(\cdot-y)|\, \dd y }}_{L^2}\\
	 &\leq \Vert{\partial^{\beta}_{\xi} a(\cdot,\xi)}\Vert_{L^p} \int  \brkt{1+2^j\abs{y}}^{-M}\, \dd y \Vert f{\partial^{\alpha-\beta}_{\xi} a(\cdot,\xi) }\Vert_{L^2}\\
	 &\leq C_{M} 2^{-jn}\Vert{\partial^{\beta}_{\xi} a(\cdot,\xi)}\Vert_{L^p} \Vert f{\partial^{\alpha-\beta}_{\xi} a(\cdot,\xi) }\Vert_{L^2},
	\end{split}
\]
provided $M>n$.
On the other hand, since $\frac{1}{2}=\frac{1}{p}+\frac{1}{r'}$,
H\"older's inequality yields
\[
	\Vert f{\partial^{\alpha-\beta}_{\xi} a(\cdot,\xi) }\Vert_{L^2}\leq \Vert f\Vert_{L^{r'}} \Vert \partial^{\alpha-\beta}_{\xi} a(\cdot,\xi)\Vert_{L^p}.
\]
Therefore, since $a\in L^p S^m_\varrho$ one has
\[
	\norm{H_M^{\alpha,\beta}f(\cdot,\xi)}_{L^{r}}\leq C_{M} \abs{a}_{p,m,\abs{\alpha-\beta}} \abs{a}_{p,m,\abs{\beta}}2^{-jn}\left< \xi\right>^{2m-\varrho\abs{\alpha}} \norm{f}_{L^{r'}},
\]
from which the result follows.
\end{proof}
\end{lem}

\begin{thm}\label{thmL2}
Let $2\leq p\leq \infty$ and define $r=\frac{2p}{p+2}$. Assume that  $\varphi\in \Phi^2$ satisfies the SND condition and let $a \in L^p S^m_\varrho$ with $0\leq \varrho\leq 1$ and 
$$m < \frac{n(\varrho-1)}{2}.$$
Then the FIO $T_a$ defined as in \eqref{definition basic linear FIO} is bounded from $L^2$ to $L^r$.
\begin{proof}
We define a Littlewood-Paley partition of unity as in \eqref{eq:LittlewoodPaley}. Let $a_j(x,\xi)=a(x,\xi)\Psi_j(\xi)$ for $j\geq 0$.  By Lemma \ref{lem:calculus_lemma}, $a_j\in L^{p}S^m_\varrho$ and for any $s\in \Z^+$ $\sup_{j\geq 0} \abs{a_j}_{p,m,s}\lesssim \abs{a}_{p,m,s}$.

That $T_{a_0}$ satisfies the required bound follows from Theorem \ref{Linearlow}, so it is enough to consider the boundedness of the operators $T_{a_j}$ for $j\geq 1$. To this end, we begin by studying the boundedness of $S_j=T_{a_j}T^{\ast}_{a_j}$.  A simple calculation yields that $S_j f(x)=\int K_j(x,y) f(y)\, \dd y$ with
 \begin{equation*}
 K_j(x,y)=\frac{1}{(2\pi)^{n}}\int e^{i(\varphi(x,\xi)-\varphi(y,\xi))} a_j(x,\xi) \overline{a_j(y,\xi)}\, \dd\xi.
 \end{equation*}
Now since $\varphi$ is homogeneous of degree $1$ in the $\xi$ variable, $K_j (x,y)$ can be written as
\begin{equation*}
K_{j}(x,y)=\frac{2^{jn}}{(2\pi)^{n}}\int m_{j}(x,y,2^{j}\xi)
e^{i2^j \Phi(x,y,\xi)}\, \dd\xi.
\end{equation*}
with $\Phi(x,y,\xi)= \varphi (x,\xi) -\varphi (y,\xi)$ and $m_j(x,y,\xi)=a_j(x,\xi) \overline{a_j(y,\xi)}$. Observe that the $\xi$-support of $m_{j}(x,y,2^{j}\xi)$ lies in the compact set $\mathcal{K}=\{\frac{1}{2}\leq \abs{\xi}\leq 2\}$.  From the mean value theorem, \eqref{C_alpha} and \eqref{lower_bound on mixed hessian}, it follows that
\begin{equation}
\label{Phi cond 1}
\vert \nabla_{\xi}\Phi (x,y, \xi)\vert \thickapprox \vert x-y\vert, \quad \text{for any $x,y\in \R^n$ and $\xi\in \mathcal{K}$}.
\end{equation}
We claim that for any $M>n$ there exists a constant $C_M$ depending only on $M$ such that
\begin{equation}\label{LPpiece}
	\norm{S_j f}_{L^{r}}\leq C_M \abs{a}_{p,m,[M]+1}^2 2^{2j m} 2^{jM(1-\varrho)} \norm{f}_{L^{r'}}, \quad \text{for any $f\in L^{r'}$},
\end{equation}
where $[M]$ stands for the integral part of $M$.

Assume first that $M>n$ is an integer.  Fix $x\neq y$ and set $\phi(\xi)=\Phi (x,y, \xi)$, $\Psi=\abs{\nabla_\xi \phi}^2$. By the mean value theorem, \eqref{C_alpha} and \eqref{Phi cond 1}, for any multi-index $\alpha$ with $\abs{\alpha}\geq 1$ and any $\xi\in \mathcal{K}$,
\[
    \abs{\d^\alpha_{\xi} \phi(\xi)}\lesssim \vert \nabla_{\xi}\Phi (x,y, \xi)\vert= \Psi^{1/2}.
\]
On the other hand, since $\d^\alpha_{\xi} \Psi=\sum_{j=1}^n \sum  \binom{\alpha}{\beta} \d^\beta_{\xi}\d_{\xi_j} \phi \d^{\alpha-\beta}_{\xi}\d_{\xi_j}  \phi$, it follows that, for any $\abs{\alpha}\geq 0$, $\abs{\d^\alpha_{\xi} \Psi}\lesssim \Psi$ and the constants are uniform on $x$ and $y$. Thus \eqref{Phi cond 1} and Lemma \ref{lem:technic} with $F=m_{j}(x,y,2^{j}\xi),$ yield
\begin{equation*}
\begin{split}
\vert K_{j}(x,y)\vert & \leq 2^{j n} 2^{-j M}\ C_{M,\mathcal{K}} \sum _{\vert \alpha \vert\leq M} 2^{j\abs{\alpha}}\int {\vert \partial^{\alpha}_{\xi} m_{j}(x,y,2^j \xi)\vert \vert \nabla_{\xi}\Phi(x,y,\xi)\vert^{-M}} \, \dd \xi\\
&\lesssim 2^{-j M} \abs{x-y}^{-M} \sum _{\vert \alpha \vert\leq M}  2^{j\abs{\alpha}} \int \abs{\partial^{\alpha}_{\xi} m_{j}(x,y,\xi)}\, \dd \xi.
\end{split}
\end{equation*}
On the other hand
\begin{equation*}
\vert K_{j}(x,y)\vert \leq \int \abs{m_j(x,y,\xi)}\, \dd \xi\lesssim \sum _{\vert \alpha \vert\leq M}  2^{j\abs{\alpha}} \int \abs{\partial^{\alpha}_{\xi} m_{j}(x,y,\xi)}\, \dd \xi.
\end{equation*}
Therefore
\begin{equation}\label{eq:Kk}
	\vert K_{j}(x,y)\vert\lesssim \brkt{1+2^j\abs{x-y}}^{-M}\sum _{\vert \alpha \vert\leq M}  2^{j\abs{\alpha}} \int \abs{\partial^{\alpha}_{\xi} m_{j}(x,y,\xi)}\, \dd \xi .
\end{equation}
Now since
\begin{equation}\label{eq:mk}
	\abs{\partial^{\alpha}_{\xi} m_{j}(x,y,\xi)}
	\leq \sum_\beta \binom{\alpha}{\beta}  \abs{\partial^{\beta}_{\xi} a_j(x,\xi)  \partial^{\alpha-\beta}_{\xi} a_j(y,\xi) },
\end{equation}
we obtain that
\begin{equation}\label{eq:S_K}
 	\abs{S_j f(x)}\leq \sum_{\vert \alpha \vert\leq M} \sum_\beta  \binom{\alpha}{\beta}   2^{j\abs{\alpha}}  \int_{\abs{\xi}\sim 2^j}  H_M^{\alpha,\beta}f(x,\xi)\, \dd \xi,
\end{equation}
where, $H_M^{\alpha,\beta}$ is defined as in \eqref{eq:H_M}. Hence, Minkowski's inequality, Lemma \ref{lem:H_M} and \eqref{eq:S_K} yield
\begin{equation}\label{eq:S_K_integers}
	\begin{split}
	\norm{S_j f}_{L^{r}}&\leq  c_M  \sum_{\vert \alpha \vert\leq M} \sum  \binom{\alpha}{\beta} \abs{a}_{p,m,\abs{\alpha}}^2 2^{j\abs{\alpha}}
	2^{-jn} \int_{\abs{\xi}\sim 2^j} \left<\xi \right>^{2m-\varrho \abs{\alpha}}\, \dd \xi  \norm{{f}}_{L^{r'}}\\
	& \leq c_M \abs{a}_{p,m,M}^2 2^{2jm}\sum_{\vert \alpha \vert\leq M} 2^{\abs{\alpha}} 2^{j\abs{\alpha}(1-\varrho)} \norm{f}_{L^{r'}}\\
	& \leq c_M \abs{a}_{p,m,M}^2 2^{2jm} 2^{jM(1-\varrho)} \norm{f}_{L^{r'}}.
	\end{split}
\end{equation}

Assume first that $M\geq n+1$ is a real number. Writing $M=[M]+\{M\}$ as the sum of its integer and fractional parts, the estimate \eqref{eq:S_K_integers} yields
\[
	\begin{split}
	\norm{S_j f}_{L^{r}} &=\norm{S_j f}_{L^{r}}^{1-\{M\}}\norm{S_j f}_{L^{r}}^{\{M\}}\leq
	\brkt{c_{[M]} \abs{a}_{p,m,[M]}^2 2^{2jm} 2^{j[M](1-\varrho)} \norm{f}_{L^{r'}}}^{1-\{M\}}\\
&\qquad\qquad\times\brkt{c_{[M]+1} \abs{a}_{p,m,[M]+1}^2  2^{2 j m} 2^{j([M]+1)(1-\varrho)} \norm{f}_{L^{r'}}}^{\{M\}}\\
	&\leq c_M  \abs{a}_{p,m,[M]+1}^2 2^{2 j m} 2^{ j M(1-\varrho)} \norm{f}_{L^{r'}}.
	\end{split}
\]

Assume now that $n<M<n+1$. Then writing $M=n+\{M\}$ and letting
\[
	R_{l}(x,y)=
	\sum _{\vert \alpha \vert\leq l}   \sum \binom{\alpha}{\beta}  2^{j\abs{\alpha}}\int_{\abs{\xi}\sim 2^{j}} \abs{\partial^{\beta}_{\xi} a(x,\xi)  \partial^{\alpha-\beta}_{\xi} a(y,\xi) }\, \dd \xi,
\]
we see that \eqref{eq:Kk} and \eqref{eq:mk} with $n$ and $n+1$ yields
\[
	\begin{split}
	\abs{K_j(x,y)} &=\abs{K_j(x,y)}^{1-\{M\}}\abs{K_j(x,y)}^{\{M\}}\\
			&\leq {R_n(x,y)}^{1-\{M\}}  {R_{n+1}(x,y)}^{\{M\}}  \brkt{1+2^j\abs{x-y}}^{-M}.
	\end{split}
\]
Hence, applying H\"older's inequality with the exponents $\frac{1}{\{M\}}$ and $\frac{1}{1-\{M\}}$ we get
\begin{multline*}
	\abs{S_j f(x)}\leq \brkt{\int R_n(x,y) \brkt{1+2^j\abs{x-y}}^{-M}\abs{f(y)}\, \dd y}^{1-\{M\}}\\
    \times \brkt{\int R_{n+1}(x,y) \brkt{1+2^j\abs{x-y}}^{-M}\abs{f(y)}\, \dd y}^{\{M\}},
\end{multline*}
and another application of the H\"older inequality with exponents $\frac{r}{\{M\}}$ and $\frac{r}{1-\{M\}}$ yields
\begin{multline*}
	\norm{S_j f}_{L^{r}}\leq \norm{\int R_n(\cdot,y) \brkt{1+2^j\abs{\cdot-y}}^{-M}\abs{f(y)}\, \dd y}^{1-\{M\}}_{L^{r}}\\
    \times
	\norm{\int R_{n+1}(\cdot,y) \brkt{1+2^j\abs{\cdot-y}}^{-M}\abs{f(y)}\, \dd y}^{\{M\}}_{L^{r}}.
\end{multline*}
Therefore, Minkowski's integral inequality and Lemma \eqref{lem:H_M} yield
\[
	\begin{split}
	\norm{S_j f}_{L^{r}}
	&\leq C_{M} \abs{a}_{p,m,n+1}^2 2^{2j m} 2^{j M(1-\varrho)} \norm{f}_{L^{r'}},
	\end{split}
\]
for all $f\in L^{r'}$.
Thus, using \eqref{LPpiece}, we obtain
\[
	\norm{T^*_{a_j} f}_{L^2}^2=\pr{f}{T_{a_j}T_{a_j}^*f}\leq \norm{f}_{L^{r'}} \norm{S_j f}_{L^{r}}\leq C_{M} \abs{a}_{p,m,[M]+1}^2 2^{2j m} 2^{j M(1-\varrho)}\norm{f}_{L^{r'}}^2,
\]
and so
\[
    \norm{T_{a_j} f}_{L^r}\leq C_{M} \abs{a}_{p,m,[M]+1} 2^{j m} 2^{j \frac{M(1-\varrho)}{2}}\norm{f}_{L^{2}},
\]
for every $f\in L^2$.

Now if $\varrho =1$ and $m<0$ we see that the sum of the Littlewood-Paley pieces $T_{a_j}$ converges and therefore $T_a$ is a bounded operator from $L^{2}$ to $L^{r}$. In case $0\leq \varrho<1$ then the condition $m<\frac{n}{2}(\varrho-1)$ implies that there is a $M_{0}$ with $n<M_{0}<\frac{-2m}{1-\varrho}$. So by choosing $M=M_0$, we have
 \begin{equation}
 \Vert T_{a_j}f\Vert_{L^{r}} \lesssim 2^{j m}2^{ j\frac{M_{0}(1-\varrho)}{2}}\Vert f\Vert_{L^{2}},
 \end{equation}
with $2m+M_{0}(1-\varrho)<0.$ This and the summation of the pieces yield the desired boundedness of $T_a$.
\end{proof}
\end{thm}

Here, we shall define a couple of parameters which will appear as the order of our operators in the remainder of this paper.
\begin{definition}\label{definition rune M} Given $1\leq p,\, q\leq \infty$ define
\[
    \textarc{m}(\varrho,p,q)=\left\{
      \begin{array}{ll}
         \frac{n(\varrho-1)}{\min(p,q)}-\frac{(n-1)}{2}\brkt{\frac{1}{p}+\frac{1}{\min(p,q)}}, & \hbox{if $1\leq p<2$, or $p\geq 2$ and $1\leq q<p'$;} \\
        \frac{n(\varrho-1)}{2}-(n-1)\brkt{\frac{1}{2}-\frac{1}{q}}, & \hbox{if $2\leq p,q$;} \\
        \frac{n(\varrho-1)}{q}-\frac{(n-1)}{1-\frac{2}{p}}\brkt{\frac{1}{q}-\frac{1}{2}}, & \hbox{if $p> 2$ and $p'\leq q\leq 2$.}
      \end{array}
    \right.
\]
Furthermore given $1<q<2$ we set
\[\mathcal{M}(\varrho,p,q)=\frac{n(\varrho-1)}{q}-\frac{n-1}{1+1/p}\brkt{\frac{1}{q}-\frac{1}{2}}. 
\]
\end{definition}
Interpolating the result of Theorem \ref{thmL2} with the extremal results of Theorem \ref{thm:Linear} using Riesz-Thorin and Marcinkiewicz interpolation theorems yields:
\begin{thm}\label{thm:p2} 
Suppose that $0<r\leq \infty$, $1\leq q \leq \infty$, $2\leq p\leq \infty$ satisfy the relation $\frac 1 r = \frac 1 q + \frac 1 p$. Assume that $\varphi\in \Phi^2$ satisfies the SND condition and let $a\in L^pS^m_\varrho$ with $0\leq \varrho \leq 1$ and $m<\textarc{m}(\varrho,p,q).$
Then  the FIO $T_a$ defined as in \eqref{definition basic linear FIO} is bounded from $L^q$ to $L^r$.
Furthermore, when $1<q<2$ and
$$
\textarc{m}(\varrho,p,q)\leq m<\mathcal{M}(\varrho,p,q),
$$
the FIO $T_a$ is bounded from $L^q$ to the Lorentz space $L^{r,q}$.
\end{thm}

\begin{remark}\label{rem:local} In theorems \ref{Linearlow}, \ref{thm:Linear}, \ref{thmL2} and \ref{thm:p2} we can replace the assumption of strong non-degeneracy of the phase function with the mere non-degeneracy condition, but then we need to add an extra assumption of compact support of the amplitude in the $x$-variable.
\end{remark}
\subsection{Boundedness of $\Psi$DOs}

A careful look at the proof of Theorem \ref{thm:Linear} reveals that in the study of rough $\Psi$DOs the Seeger-Sogge-Stein frequency decomposition is unnecessary, and it suffices to  use a Littlewood-Paley decomposition.  Therefore with minor modifications, the proof  of the aforementioned theorem carries over to the case of $\Psi$DOs, with the difference that there will be no contribution (loss of derivatives) due to the existence on a non-linear phase function in the operator.  So, we obtain the following result  which extends those in \cite{KenStaub, MRS}. The details are left to the interested reader.
\begin{thm}\label{thm:LinearPDO}
 Suppose that $0<r\leq \infty$, $1\leq p,q \leq \infty$, satisfy the relation $\frac 1 r = \frac 1 q + \frac 1 p$. Let $a \in L^p S^m_\varrho$ with $0\leq \varrho \leq 1$ and
\begin{equation*}
	m < \frac{n(\varrho-1)}{\min(2,p,q)}.
\end{equation*}
Then the pseudodifferential operator $T_a f(x)=\int a(x,\xi) e^{i\p{x,\xi}} \widehat{f}(\xi)\dd \xi$ is bounded from $L^q$ to $L^r$.
\end{thm}

\begin{remark}Observe that for $p=\infty$ and $q=2$ the result is sharp (see \cite{KenStaub}).
\end{remark}
\section{Applications to the boundedness of multilinear Pseudodifferential and Fourier integral operators}\label{bilinear PsiDo and FIO}
Before we state and prove the boundedness results for multilinear operators, we shall define the classes of symbols (or amplitudes) that we are concerned with in this paper. The following definitions, given for a fixed natural number $N\geq 1$, extend definitions \ref{definition of hormander amplitudes} and \ref{LpSmrho definition}.
\begin{definition}\label{definition multilinear hormander amplitudes}
Given $m\in \R$ and $\varrho, \delta \in [0,1],$ the amplitude $a(x,\xi_1,\dots,\xi_N)\in C^{\infty} (\R^n \times \R^{Nn})$ belongs to the multilinear H\"ormander class $S^{m}_{\varrho, \delta} (n, N)$ provided that for all multi-indices $\beta$, $\alpha_1,\dots,\alpha_N$ in $\Z_{+}^{n}$ it verifies
\begin{equation}
    \abs{\partial^\beta_x \partial_{\xi_1}^{\alpha_1}\dots\partial_{\xi_N}^{\alpha_N} a(x,\xi_1,\dots,\xi_N)}\leq C_{\alpha_1,\dots,\alpha_N,\beta} \brkt{1+\abs{\xi_1}+\dots+\abs{\xi_N}}^{m-\varrho\sum_{j=1}^{N}\abs{\alpha_j}+\delta\abs{\beta}}.
\end{equation}

\end{definition}

We shall also use the classes of non-smooth amplitudes one of which is defined as follows:

\begin{definition}\label{definition multilinear product type amplitudes} Let $\vec m=(m_1,\dots, m_N)\in \R^{N}$ and ${\vec \varrho}=(\varrho_1,\dots,\varrho_N)\in [0,1]^N$. The amplitude $a(x,\xi_1,\ldots, \xi_N)$, defined for $x,\xi_1,\ldots,\xi_N\in \R^n$, belongs to the class $L_{\Pi}^p S^{\vec m}_{\vec \varrho}(n,N)$ if for any multi-indices $\alpha_1,\dots,\alpha_N$  in $\Z_{+}^{n}$,  there exists a constant $C_{\alpha_1, \dots, \alpha_N}$ such that
\begin{equation}\label{eq:Lp}
    \norm{\partial_{\xi_1}^{\alpha_1}\dots\partial_{\xi_N}^{\alpha_N} a(\cdot,\xi_1,\dots, \xi_N)}_{L^p}\leq C_{\alpha_1, \dots, \alpha_N} \prod_{j=1}^N \p{\xi_j}^{m_j-\varrho_j\abs{\alpha_j}}.
\end{equation}
\end{definition}

We remark that the subscript $\Pi$ in the notation $ L_{\Pi}^p S^{\vec m}_{\vec \varrho}(n,2)$ is there to indicate the product structure of these type of amplitudes.
 From now on, we shall fix $N\geq 2$. 
\begin{example}
Any amplitude in the class $S^{0}_{1,1}(n,N)$ considered by Grafakos and Torres in \cite{GrafTor}, belongs to $L_\Pi ^\infty S^{(0,\dots,0)}_{(1,\dots, 1)}(n,N)$.
\end{example}

\begin{example}
Let $a_j (x, \xi_j)\in L^{p_j}S^{m_j}_{\varrho_j}$ for $j=1,\dots, N$, be a collection of linear amplitudes and assume that $\frac{1}{p}=\sum_{j=1}^{N}\frac{1}{p_j}.$ Then the multilinear amplitude
$\prod_{j=1}^{N}a_j (x, \xi_j)$ belongs to the class $L_\Pi ^pS^{(m_1,\ldots,m_N)}_{(\varrho_1,\ldots,\varrho_N)}(n,N).$
\end{example}

Also we have the following class of non-smooth amplitudes introduced in \cite{MRS}.
\begin{definition}\label{definition multilinear MRS amplitudes} The amplitude $a(x,\xi_1,\ldots,\xi_N)$, defined for $x,\xi_1,\ldots,\xi_N\in \R^n$, belongs to the class $L^p S^{m}_{\varrho}(n,N)$ if for any multi-indices $\alpha_1,\dots, \alpha_N$ there exists a constant $C_{\alpha_1, \dots, \alpha_N}$ such that
\begin{equation}
    \norm{\partial_{\xi_1}^{\alpha_1}\dots\partial_{\xi_N}^{\alpha_N} a(\cdot,\xi_1,\dots ,\xi_N)}_{L^p}\leq C_{\alpha_1, \dots, \alpha_N} \brkt{1+\abs{\xi_1}+\dots+\abs{\xi_N}}^{m-\varrho\sum_{j=1}^{N}\abs{\alpha_j}}.
\end{equation}
\end{definition}

\begin{example}\label{interesting example}
It is easy to see that if $m\leq 0,$ $m_j\leq 0$ for $j=1, \dots, N$ and $p\in [1, \infty],$ then
\[
    L^p S^{m}_\varrho(n,N)\subset \bigcap_{m_1+\dots +m_N=m} L^p_\Pi S^{(m_1,\dots,m_N)}_{(\varrho,\dots, \varrho)}(n,N).
\]
Moreover, for any $\varrho,\delta\in [0,1]$, $S^{m}_{\varrho, \delta}(n,N)\subset L^\infty S^{m}_\varrho(n,N)$ .
\end{example}

\subsection{Boundedness of multilinear FIOs}
In this section we shall apply the boundedness of the linear FIOs obtained in the previous section to the problem of boundedness of bilinear and multilinear operators.

To any class of the amplitudes defined above we associate the corresponding multilinear Fourier integral operator given by
\begin{equation}\label{eq:definition_FIO}
T_a(f_1, \dots, f_N) (x)= \int_{\R^{Nn}} a(x,\xi_1, \dots, \xi_N)\, e^{ i\sum_{j=1}^{N} \varphi_{j}(x, \xi_{j})} \, \prod_{j=1}^{N}\widehat{f}(\xi_j)\, \dd \xi_1\dots \dd \xi_N.
\end{equation}

At this point, for the sake of simplicity of the exposition and until further notice, we confine ourselves to the study of boundedness of bilinear operators. Here, using an iteration procedure, we are able to reduce the problem of global boundedness of bilinear FIOs to that of boundedness of rough and linear FIOs. Our main result in this context is as follows.

\begin{thm}\label{thm: bilinear product type operators} Suppose that $0<r\leq \infty$, $1\leq p,q_1,q_2 \leq \infty$, satisfy that $\frac{1}{r}=\frac{1}{p}+\frac{1}{q_1}+\frac{1}{q_2}$ and $q_1=\max(q_1,q_2)\geq p'$. Assume that $\varphi_1,\varphi_2\in \Phi^2$ satisfy the SND condition and let $a\in L_\Pi ^{p}S^{(m_1,m_2)}_{(\varrho_1 , \varrho_2)}(n,2)$ with $0\leq \varrho_1 , \varrho_2\leq 1$ and 
\[
    m_1<\textarc{m}(\varrho_1,p,q_1) \quad {\rm and} \quad m_2<\textarc{m}(\varrho_2,r_2,q_2),
\]
with $\frac{1}{r_2}=\frac{1}{p}+\frac{1}{q_1}.$ Then the bilinear FIO $T_a$, defined by
\begin{equation}
\label{Intro:bilinear_Fourier integral operator}
 T_{a} (f,g)(x)= \iint a(x,\xi,\eta)\, e^{ i \varphi_{1}(x, \xi)+ i\varphi_{2}(x, \eta)} \hat{ f}(\xi) \hat{g}(\eta)\,\dd \xi\, \dd \eta,
\end{equation}
satisfies the estimate
\[
    \norm{T_a(f,g)}_{L^r}\leq C_{a,n} \norm{f}_{L^{q_1}}\norm{g}_{L^{q_2}},\quad \text{for any $f,\,g\in \S$}.
\]
Moreover, if $1\leq q_2<2\leq r_2$,
\[
   m_1<\textarc{m}(\varrho_1,p,q_1) \quad {\rm and} \quad \textarc{m}(\varrho_2,r_2,q_2)\leq m_2<\mathcal{M}(\varrho_2,r_2,q_2),
\]
then
\[
    \norm{T_a(f,g)}_{L^{r,q_2}}\leq C_{a,n} \norm{f}_{L^{q_1}}\norm{g}_{L^{q_2}},\quad \text{for any $f,\,g\in \S$}.
\]
\end{thm}
\begin{proof}
For any $f,\, g\in \S$ set
$
    a_f\brkt{x,\eta}=\int e^{i\varphi_1(x,\xi)} a(x,\xi,\eta) \widehat{f}(\xi)\, \dd \xi.
$
Observe that the amplitude $\partial^{\alpha}_{\eta}a(\cdot,\xi,\eta)\in L^pS^{m_1}_{\varrho_1}$ if $\eta $ is hold fixed, and moreover for any $s\in \Z_+$,
\[
    \abs{\partial^{\alpha}_{\eta}a(\cdot,\cdot,\eta)}_{m_1,p,s}\leq c_{\alpha,s}\p{\eta}^{m_2-\varrho_2 \abs{\alpha}}.
\]
Thus, depending on the range of indices, we apply Theorem \ref{thm:Linear} or Theorem \ref{thm:p2} to obtain
\[
    \norm{\partial^{\alpha}_{\eta}a_f\brkt{\cdot,\eta}}_{L^{r_2}}\lesssim \p{\eta}^{m_2-\varrho_2 \abs{\alpha}} \norm{f}_{L^{q_1}},
\]
provided $m_1<\textarc{m}(\varrho_1,p,q_1)$. This means that $a_f\in L^{r_2}S^{m_2}_{\varrho_2}$ and for any $s\in \Z_+$,
\[
    \abs{a_f}_{r_2,m_2,s}\lesssim \norm{f}_{L^{q_1}}.
\]
Now applying either Theorem \ref{thm:Linear} or Theorem \ref{thm:p2} again, we obtain the desired result.
\end{proof}

\begin{cor}\label{cor: cor of the main bilinear}  Suppose that $0<r\leq \infty$, $1\leq q_1,q_2\leq \infty$ satisfy the relation $\frac{1}{r}=\frac{1}{q_1}+\frac{1}{q_2}$. Let $q_{\rm max}=\max(q_1,q_2)$ and $q_{\rm min}=\min(q_1,q_2)$.
Assume that  $\varphi_1,\varphi_2\in \Phi^2$ satisfy the SND condition and let $a\in L^\infty S^{m}_\varrho(n,2)$ with $0\leq \varrho\leq 1$ and 
\[
    m<\textarc{m}\brkt{\varrho,\infty,q_{\rm max}}+\textarc{m}\brkt{\varrho,q_{\rm max},q_{\rm min}}.
\]
Then the bilinear FIO $T_a$ defined by \eqref{Intro:bilinear_Fourier integral operator} satisfies
\[
    \norm{T_a(f,g)}_{L^r}\leq C_{a,n} \norm{f}_{L^{q_1}}\norm{g}_{L^{q_2}}, \quad \text{for any $f,\,g\in \S$.}
\]
Moreover, if $1\leq q_{\rm min}<2\leq q_{\rm max}$ and
\[
  \textarc{m}(\varrho,\infty,q_{\rm max})+ \textarc{m}(\varrho,q_{\rm max,}q_{\rm min})\leq m<\textarc{m}(\varrho,\infty,q_{\rm max})+\mathcal{M}(\varrho,q_{\rm max},q_{\rm min}),
\]
then
\[
    \norm{T_a(f,g)}_{L^{r,q_{\rm min}}}\leq C_{a,n} \norm{f}_{L^{q_1}}\norm{g}_{L^{q_2}}, \quad \text{for any $f,\,g\in \S$.}
\]
\end{cor}

\begin{remark}
In the case $\varrho=1,\, r=1$ this corollary yields a global bilinear $L^2\times L^2\to L^1$ extension of H\"ormander and Eskin's local $L^2$ boundedness of zeroth order linear FIOs (see  \cite{Eskin, Hor2}).

 L. Grafakos and M. Peloso \cite{GrafPelo} proved the local  $L^{1} \times L^\infty \to L^{1}$  and  $L^\infty\times L^{1} \to L^{1}$ boundedness of bilinear FIOs with non-degenerate phase functions of the form  $\phi(x,\xi, \eta)$  and amplitudes in $S^m_{1,0}(n,2)$ under the assumption $m<-n + \frac{1}{2}.$ Our Corollary \ref{cor: cor of the main bilinear} yields the global  $L^{p} \times L^\infty \to L^{p}$  and  $L^\infty\times L^{p} \to L^{p}$  ($1\leq p\leq\infty$) boundedness of bilinear FIOs whose phase function is of the form $\varphi_1(x,\xi)+ \varphi_2(x,\eta)$,  provided that $m< -(n-1) (\frac{1}{2}+|\frac{1}{p} -\frac{1}{2}|)$.
\end{remark}

Here we would like to mention how Theorem \ref{thm:p2} above can be used to extend a couple of known results in the literature. In \cite{GrafPelo}  Grafakos and Peloso showed the following theorem:
\begin{thm}\label{thm: grafakos-peloso}
Let $a(x, \xi,\eta) \in S^{m}_{1,0} (n,2)$ with compact support in the spatial variable $x$. The the corresponding bilinear FIO defined by \eqref{eq:definition_FIO} with phase functions $\varphi_1(x,\xi)$ and $\varphi_2(x,\eta)$  satisfying the non-degeneracy condition,  is bounded from $L^{q_1}\times L^{q_2} \to L^{r}$ with $\frac{1}{q_1} +\frac{1}{q_2} =\frac{1}{r}$ and $1\leq q_{1}, \, q_{2} \leq 2,$ provided that the order $m<-(n-1)\left((\frac{1}{q_1}-\frac{1}{2})+(\frac{1}{q_2}-\frac{1}{2})\right).$
\end{thm}

Furthermore, F. Bernicot  \cite{Bernicot} obtained the following sharp proposition concerning bilinear multipliers (which are in turn a subclass of bilinear pseudodifferential operators):

\begin{prop} \label{prop: bernicot}
$\sigma$ be a multiplier in $\sigma\in S^{(0,0)}_{(1,1)}(n,2)$, i.e. a bounded function on $\mathbb{R}^2$ such that
$$\forall \alpha ,\beta \geq 0\qquad \left|\partial^{\alpha}_{\xi} \partial^{\beta}_{\eta} \sigma(\xi,\eta)\right| \lesssim \left(1+|\xi|\right)^{-|\alpha|}\left(1+|\eta|\right)^{-|\beta|}.$$
Let $1<p,q \leq \infty$ be exponents such that $0<\frac{1}{r}=\frac{1}{q_1}+\frac{1}{q_2}.$
Then, for any $s>0$, the associated bilinear pseudodifferential operator $T_\sigma$ is continuous from $H^{s,q_1} \times H^{s,q_2}$ to $L^r$.
\end{prop}
Here, the Sobolev space $H^{s,p}$ is the space of tempered distributions $f$ such that $(1-\Delta)^{\frac{s}{2}}f (x)$ belongs to $L^p$.

 The following theorem yields the sharp boundedness of a rather large class of rough multilinear Fourier integral operators on $L^{r}$ spaces for $0<r\leq \infty$. Apart from the global multilinear generalization of Theorem \ref {thm: grafakos-peloso} above, it also extents it to the class of rough symbols with product type structure and all ranges of $q_{1}$'s  and $q_{2}$'s.

Moreover, our result also yields a generalization of Proposition \ref{prop: bernicot} to the case of rough multilinear FIOs.   In the case of operators defined with phase functions that are inhomogeneous in the $\xi$-variable, we are also able to show a boundedness result in case the multilinear operator acts on $L^2$ functions.

\begin{thm}\label{Thm:generalmultilinear} Let $1\leq p\leq \infty$, $m_{j}<0,$ $j= 1,\dots N,$ and suppose that $\frac{\sum_{j=1}^{N} m_{j}}{\min_{j=1,\dots,N}m_{j}}\geq \frac{2}{p}$. Assume that the amplitude $a(x,\xi_{1},\dots, \xi_{N})\in L^{p}_{\Pi} S^{(m_1, \dots, m_N)}_{(1,\dots, 1)}(n,N)$ and the phase functions $\varphi_j\in \Phi^2,$ $j=1, \dots, N,$ satisfy the SND condition and belong to the class $\Phi^2.$

For $1\leq q_j < \infty$ in case $p=\infty,$ and $1\leq q_j \leq \infty$ in case $p\neq \infty,$ $j=1,\dots,N,$ let
\[
    \frac{1}{r}=\frac{1}{p}+\sum_{j=1}^{N}\frac{1}{q_j}.
\]
Then the multilinear FIO
\begin{equation}\label{eq:multilinearFIO}
T_a(f_1, \dots, f_N) (x)= \int_{\R^{Nn}} a(x,\xi_1, \dots, \xi_N)\, e^{ i\sum_{j=1}^{N} \varphi_{j}(x, \xi_{j})} \, \prod_{j=1}^{N}\hat{f}(\xi_j)\,\dd \xi_1\dots \dd \xi_N,
\end{equation}
satisfies the estimate
\[
    \norm{T_a(f_1,\dots, f_N)}_{L^r}\leq C_{a,n} \norm{f_1}_{L^{q_1}}\dots\norm{f_N}_{L^{q_N}}, 
\]
provided that
\[
    m_j<\textarc{m}(1,\frac{p(\sum_{k=1}^{N}m_k)}{m_j},q_j),\quad {\rm for}\, j=1, \dots , N.
\]

Furthermore, $T_{a}$ with $a\in L^\infty S^m_{1}(n,N)$ is bounded from $L^2 \times \dots\times L^2 \to L^{\frac{2}{N}}$ provided that $m<0$ and the phases $\varphi_{j}\in C^{\infty}(\R^n \times \R^n)$ satisfy the SND condition and $|\partial_{x}^{\alpha} \partial^{\beta}_{\xi} \varphi_{j} (x, \xi)|\leq C_{j,\alpha,\beta}$ for $j=1,\dots, N$ and all multi-indices $\alpha $ and $\beta$ with $2\leq |\alpha|+|\beta|.$ Note in this case, we do not require any homogeneity from the phase functions.
\begin{proof}
We will only give the proof of the theorem in the case of bilinear operators, since using the well-known inequality
$\sum_{j\geq 0}\prod_{1\leq k\leq N} |a_{j,k}| \leq \prod_{1\leq k\leq N} \brkt{\sum_{j\geq 0} |a_{j,k}|^2}^{\frac{1}{2}}$
and the H\"older inequality in \eqref{eq:bilinear} below yield the result in the multilinear case.

Let us start by clarifying the sharpness of the theorem in the case $p=\infty$. The reason behind this claim is a result of L. Grafakos and N. Kalton \cite {GrafKal}, which states that there are bilinear symbols in the class  $S^{(0, 0)}_{(1, 1)}(n,2)$ for which the associated bilinear pseudodifferential operator is unbounded. Since bilinear $\Psi$DOs are just a special case of bilinear FIOs, and the class $S^{(0, 0)}_{(1, 1)}(n,2)$  is a subclass of  $L^\infty_\Pi S^{(0, 0)}_{(1,1)}(n,2)$, we can deduce the desired sharpness.

Now let $\{\Psi_j\}_{j\geq 0}$ a Littlewood-Paley partition of unity in $\R^{2n}$ as in \eqref{eq:LittlewoodPaley}. Let for $j\geq 0$, $a_j(x,\xi, \eta)=a(x,\xi, \eta)\Psi_j(\xi, \eta)$.
Then we have
\[
    T_a(f,g)(x)=\sum_{j\geq 0} 2^{2jn} \iint a_{j}(x,2^{j}\xi, 2^{j} \eta) \widehat{f}(2^{j}\xi)\widehat{g}(2^{j}\eta)e^{i\varphi_{1}(x,2^{j}\xi)+i\varphi_{2}(x,2^{j}\eta)}\, \dd \xi \, \dd \eta.
\]
 Now since for any $j\geq 0$, $\Psi_j(2^j\xi, 2^{j} \eta)$ is supported in $B(0,2)\subset \T^{2n}$, following the argument in Theorem \ref{Linearlow} and expanding the amplitudes in Fourier series, we obtain
\[
    a_{j}(x,2^{j}\xi, 2^{j} \eta)=\sum_{(k,l)\in \Z^{2n}} a_{k,l}^j(x) e^{i\p{k,\xi}+i\p{l,\eta}}.
\]
Moreover, for any natural number $M\geq 1$, we have
\[
    \abs{a_{k,l}^j(x)}\lesssim \frac{1}{1+\abs{(k,l)}^M} \sum_{\alpha_1 + \alpha_{2}=\alpha;\,\abs{\,\alpha}\leq M}\int_{\T^{2n}} \abs{\partial^{\alpha_1}_\xi \partial^{\alpha_2}_\eta \brkt{a_{j}(x,2^j\xi, 2^{j} \eta)}}\, \dd \xi\, \dd \eta,
\]
and
\begin{equation}
\begin{split}
    \Vert a_{k,l}^j\Vert_{L^{p}}&\leq  \frac{1}{1+\abs{(k,l)}^M}  \sum_{\alpha_1 + \alpha_{2}=\alpha;\,\abs{\,\alpha}\leq M}\int_{\T^{2n}} \Vert \partial^{\alpha_1}_\xi \partial^{\alpha_2}_\eta a_{j}(x,2^j\xi, 2^{j} \eta)\Vert_{L^{p}_{x}}\, \dd \xi\, \dd \eta\\&\lesssim \frac{2^{j(m_1+m_2)}}{1+\abs{(k,l)}^M}.
\end{split}
\end{equation}
We now take a $\zeta\in C_{0}^{\infty}(\R^{n})$ equal to one in the cube $[-3,3]^n$ and such that $\supp \zeta\subset \T^{n}.$ This yields
\[
    \begin{split}
    &T_a(f,g)(x)\\&=\sum_{(k,l)\in \Z^{2n}} \sum_{j\geq 0} a_{k,l}^j(x)\iint\zeta(2^{-j}\xi)\zeta(2^{-j}\eta) e^{i\p{2^{-j}k,\xi}}e^{i\p{2^{-j}l,\eta}}\widehat{f}(\xi)\widehat{g}(\eta)e^{i\varphi_{1}(x,\xi)+i\varphi_{2}(x,\eta)}\, \dd \xi\, \dd \eta\\
    &=\sum_{(k,l)\in \Z^{2n}} \sum_{j\geq 0} {\rm sgn\,}\brkt{a_{k,l}^{j}(x)}T_{\theta^{j,1}_{k,l},\varphi_1}(f)(x)T_{\theta^{j,2}_{k,l},\varphi_2}(g)(x),
    \end{split}
\]
where  ${\rm sgn\,}z={z}/{\abs{z}}$ if $z\neq 0$ and zero elsewhere, $\theta^{j,1}_{k,l}(x,\xi)=\abs{a_{k,l}^j(x)}^{\frac{m_1}{m_1+m_2}}\zeta(2^{-j}\xi)e^{i\p{2^{-j}k,\xi}},$  $\theta^{j,2}_{k,l}(x,\eta)=\abs{a_{k,l}^j(x)}^{\frac{m_2}{m_1+m_2}}\zeta(2^{-j}\eta)e^{i\p{2^{-j}l,\eta}}$, and $T_{\theta^{j,1}_{k,l},\varphi_1}$ (resp. $T_{\theta^{j,2}_{k,l},\varphi_2}$) stands for the FIO with amplitude $\theta^{j,1}_{k,l}$ (resp. $\theta^{j,2}_{k,l}$) and phase function $\varphi_1$ (resp. $\varphi_2$). 
If we let $R=\min (1,r),$ then the Cauchy-Schwarz and the H\"older inequalities yield
\begin{equation}\label{eq:bilinear}
    \norm{T_a(f,g)}_{L^r}^{R}\leq \sum_{(k,l)\in \Z^{2n}}
    \norm{\brkt{\sum_{j\geq 0} \abs{T_{\theta^{j,1}_{k,l},\varphi_1}(f)(x)}^2}^{\frac{1}{2}} }_{L^{r_1}}^{R}
    \norm{\brkt{\sum_{j\geq 0} \abs{T_{\theta^{j,2}_{k,l},\varphi_2}(g)(x)}^2}^{\frac{1}{2}}}_{L^{r_2}}^{R},
\end{equation}
where $\frac{1}{r}=\frac{1}{r_1}+\frac{1}{r_2}$, and
\[
    \frac{1}{r_1}=\frac{1}{q_1}+\frac{m_1}{p(m_1+m_2)},\qquad  \frac{1}{r_2}=\frac{1}{q_2}+\frac{m_2}{p(m_1+m_2)}.
\]
Khinchine's inequality yields
\[
     \norm{\brkt{\sum_{j\geq 0} \abs{T_{\theta^{j,1}_{k,l},\varphi_1}(f)(x)}^2}^{\frac{1}{2}} }_{L^{r_1}(\R^n)}\lesssim
     \norm{\sum_{j\geq 0 } \varepsilon_j (t) T_{\theta^{j,1}_{k,l},\varphi_1}(f)(x)}_{L_{x,t}^{r_1}(\R^n\times [0,1])},
\]
where $\{\varepsilon_j (t)\}_j$ are the Rademacher functions.
Observe that the inner term is a linear FIO with the phase function $\varphi_1$ and the amplitude
\[
    \sigma^1_{k,l}(t,x,\xi)=\sum_{j\geq 0 } \varepsilon_{j}(t)\theta^{j,1}_{k,l}(x,\xi),\quad t\in [0,1],\, x,\xi\in \R^n.
\]
Picking $s_1,s_2$ such that $m_i<s_i<\textarc{m}(1,\frac{p(m_1+m_2)}{m_i},q_i)$, for $i=1,2,$ then since $\supp \zeta\subset B(0,4\sqrt{n})$, one can see that for any multi-index $\alpha$
\[
    \abs{\d^\alpha_{\xi}\brkt{\varepsilon_j (t)\zeta(2^{-j}\xi)e^{i\p{2^{-j}k,\xi}}}}\lesssim \p{\xi}^{s_1-\abs{\alpha}}\brkt{1+\abs{k}^{\abs{\alpha}}}2^{-js_1},
\]
with a constant which is uniform in $j$ and $t$. In particular, $ \sigma^1_{k,l}\in L^{\frac{p(m_1+m_2)}{m_1}} S^{s_1}_1$ and
\[
    \norm{\d^\alpha_{\xi} \sigma^1_{k,l}(t,x,\xi)}_{L^{\frac{p(m_1+m_2)}{m_1}}}\lesssim \frac{\p{\xi}^{s_1-\abs{\alpha}}\brkt{1+\abs{k}^{\abs{\alpha}}}}{\brkt{1+\abs{k}^M}^{\frac{m_1}{m_1+m_2}}}.
\]
By the hypotheses on $m_1,m_2,$ we have that $\frac{p(m_1+m_2)}{m_1}\geq 2$, and therefore Theorem \ref{thm:p2} yields
\[
    \norm{\sum_{j\geq 0 } \varepsilon_j (t) T_{\theta^{j,1}_{k,l},\varphi_1}(f)(x)}_{L^{r_1}(\R^n\times [0,1])}\lesssim \frac{1+\abs{k}^{M_1}}{\brkt{1+\abs{k}^M}^{\frac{m_1}{m_1+m_2}}} \norm{f}_{L^{q_1}},
\]
for a certain natural number $M_1.$
Arguing in the same way with the second term of \eqref{eq:bilinear} we have
\[
    \norm{T_a(f,g)}_{L^r}^R\leq \sum_{(k,l)\in \Z^{2n}} \brkt{\frac{\brkt{1+\abs{k}^{M_1}}\brkt{1+\abs{l}^{M_2}}}{{1+\abs{(k,l)}^M}}}^R \norm{f}_{L^{q_1}}^R\norm{g}_{L^{q_2}}^R.
\]
Therefore by choosing $M$ large enough, we obtain the desired boundedness result.

The last assertion is a direct consequence of the method of proof of the first claim, and the $L^2$ boundedness of oscillatory integral operators with amplitudes in $S^{0}_{0,0}$ and strongly non-degenerate inhomogeneous phase functions satisfying the hypotheses of our theorem, which is due to K. Asada and D. Fujiwara \cite{AsadFuji}. The proof of the theorem is therefore concluded.
\end{proof}
\end{thm}

\begin{remark} The theorem above does not cover the cases $p=\infty$ or when at least one $q_j=\infty$, but Theorem \ref{thm: bilinear product type operators} fills that gap at least in the bilinear case.
\end{remark}

An immediate consequence of the theorem is the following extensions of Theorem \ref{thm: grafakos-peloso} and Proposition \ref{prop: bernicot} due to Grafakos and Peloso\cite{GrafPelo} and Bernicot \cite{Bernicot} respectively. 

\begin{cor} With the same assumptions as in Theorem \ref{Thm:generalmultilinear}, we have that the multilinear FIO $T_a$ is continuous from $H^{s_1,q_1} \times \ldots\times H^{s_N,q_N}$ to $L^r$ for every $s_j>0$ for $j=1,\ldots,N$, provided that
$m_j\leq \textarc{m}(1,\frac{p(\sum_{k=1}^{N}m_k)}{m_j},q_j)$ for $j=1,\dots, N$, 
\begin{proof} It suffices to observe that, for $s_j>0$ then 
\[
	a(x,\xi_1,\ldots,\xi_N)\prod_{j=1}^N \p{\xi_j}^{-s_j}\in L^p_\Pi S^{(m_1-s_1,\ldots, m_N-s_N)}_{(1,\ldots,1)}(n,N),
\]
and apply Theorem \ref{Thm:generalmultilinear}.
\end{proof}
\end{cor}
\begin{cor}  Let $0<r<\infty$, $1\leq q_1,\ldots, q_N<\infty$ satisfying $\frac{1}{r}=\sum_{j=1}^{N}\frac{1}{q_j}$. Assume that $\varphi_j\in \Phi^2,$ $j=1, \dots, N,$ satisfy the SND condition and belong to the class $\Phi^2$. If $a\in L^\infty S^m_1(n,N)$ with
\[
    m<-(n-1)\sum_{j=1}^{N}\abs{\frac{1}{q_j}-\frac{1}{2}},
\]
then the multilinear FIO given by \eqref{eq:multilinearFIO} satisfies the estimate
\[
    \norm{T_a(f_1,\dots, f_N)}_{L^r}\leq C_{a,n} \norm{f_1}_{L^{q_1}}\dots\norm{f_N}_{L^{q_N}}.
\]
\begin{proof}
Indeed if $p=\infty$ and $\varrho=1$ then $\textarc{m}(1,\infty ,q_j)= -(n-1)|\frac{1}{q_{j}}-\frac{1}{2}|$ and since according to Example \ref{interesting example}, $L^{\infty} S^{m}_1(n,N)\subset \bigcap_{m_1+\dots+m_N=m} L^{\infty}_\Pi S^{(m_1,\dots,m_N)}_{(1,\dots,1)}(n,N)$, for $m_j <0$, the previous theorem yields the result.
\end{proof}
\end{cor}

\subsubsection{Boundedness of smooth bilinear FIOs}
\begin{thm}\label{thm:smoothbilinearFIO} Let $1< q_1,q_2<\infty$ and  $0<\frac{1}{r}=\frac{1}{q_1}+\frac{1}{q_2}$. Assume that  $\varphi_1,\varphi_2\in \Phi^2$ satisfy the non-degeneracy condition and let $a\in S^{m}_{\varrho,1-\varrho}(n,2)$ with $\frac{1}{2} \leq \varrho\leq 1$, is compactly supported in the $x$-variable and
\[
    m<\min\brkt{(\varrho-n)\abs{\frac{1}{q_1}-\frac{1}{2}}+\textarc{m}(\varrho,q_1,q_2),(\varrho-n)\abs{\frac{1}{q_2}-\frac{1}{2}}+\textarc{m}(\varrho,q_2,q_1)}.
\]
Then the bilinear FIO $T_a$ defined by \eqref{Intro:bilinear_Fourier integral operator} satisfies
\begin{equation}\label{eq:boundedness}
    \norm{T_a(f,g)}_{L^r}\leq C_{a,n} \norm{f}_{L^{q_1}}\norm{g}_{L^{q_2}}, \quad \text{for every $f,\,g\in \S$}.
\end{equation}

A global $L^2\times L^2$ to $L^1$ boundedness result is valid for $\varphi_1,\varphi_2\in \Phi^2$ satisfying the SND condition and $a\in S^{m}_{\varrho,\delta}(n,2)$ with  $0\leq \varrho,\delta\leq 1$ and
\[
	m<\frac{n(\varrho-1)}{2}+\frac{n\min(\varrho-\delta,0)}{2}.
\]
\begin{proof} Let $m$ be as in the statement of the theorem. Let $\chi\in C^\infty (\R)$ such that $0\leq \chi\leq 1$, $\supp \chi \in [-2,2]$ and $\chi(s)+\chi(1/s)=1$ for $s>0$. Define
\[
	a_1(x,\xi,\eta)=a(x,\xi,\eta) \chi\brkt{\frac{\p{\xi}^2}{\p{\eta}^2}},\quad  a_2(x,\xi,\eta)=a(x,\xi,\eta) \chi\brkt{\frac{\p{\eta}^2}{\p{\xi}^2}} .
\]
It is easy to see that $a_1,a_2\in S^{m}_{\varrho, 1-\varrho}(n,2)$ and $a_1+a_2=a$. So, it suffices to prove that $T_{a_1}$ and $T_{a_2}$ satisfy $\eqref{eq:boundedness}$.  Observe that since $a_1$ is supported in the region $\abs{\xi}^2\leq 1+2\abs{\eta}^2$, one has for any multi-indices $\alpha,\beta, \gamma$,
\[
	\abs{\d_x^\gamma \d_\xi^\alpha \d_\eta^\beta a_1(x,\xi,\eta)}\lesssim \p{\xi}^{m_1-\varrho\abs{\alpha}}\p{\eta}^{m_2-\varrho\abs{\eta}+(1-\varrho)\abs{\delta}},
\]
for $m=m_1+m_2$ with $m_1<\textarc{m}(\varrho,q_1,q_2)$ and $m_2<(\varrho-n)\abs{\frac{1}{q_1}-\frac{1}{2}}$. In particular, if we fix $\xi$, $\d^{\beta}_\xi a_2(x,\xi,\eta)\in S^{m_2}_{\varrho,1-\varrho}$ and has compact spatial support. Therefore for fixed  $g\in \S$, \cite[Theorem 5.1]{SSS} implies that
\[
	a_{g}(x,\xi)=\int a(x,\xi,\eta) e^{i\varphi_2(x,\eta)} \widehat{g}(\eta)\dd \eta,
\]
belongs to $L^{q_2}S^{m_1}_{\varrho}$ and  $\norm{\d_\xi^\beta a_{g}(\cdot,\eta)}_{L^{q_2}}\lesssim \p{\xi}^{m_1-\varrho\abs{\beta}} \norm{g}_{L^{q_1}}$. Then, applying  Theorem \ref{thm:Linear} or  \ref{thm:p2}, and taking into account Remark \ref{rem:local}, we obtain that $T_{a_2}$ satisfies \eqref{eq:boundedness}. The boundedness of $T_{a_1}$ is proved in a similar way.

The last statement is proved in an analogous way, by an iteration argument using \cite[Theorem 2.4.1]{DosStaub} in the first step instead of \cite[Theorem 5.1]{SSS}, therefore we omit the details.
\end{proof}
\end{thm}
\subsection{Boundedness of bilinear $\Psi$DOs}\label{bilinear PSDO}

In this section we shall apply the boundedness of the linear $\Psi$DOs to the problem of boundedness of bilinear operators.

Following the same idea as in the proof of Theorem \ref{thm: bilinear product type operators} and using Theorem \ref{thm:LinearPDO} we can obtain results for bilinear pseudodifferential operators. In particular, this yields the following extension of  \cite[Thm. 3.3]{MRS} for the case $w=\mu=1$.

\begin{thm} Let $1\leq p, q_1,q_2\leq \infty$ satisfying the relation 
$
    \frac{1}{r}=\frac{1}{p}+\frac{1}{q_1}+\frac{1}{q_2}.
$
Let $q_{\rm max}=\max(q_1,q_2)$ and $q_{\rm min}=\min(q_1,q_2)$. Assume that $q_{\rm max} \geq p'$ and let $a\in L^p S^m_\varrho(n,2)$ with $0\leq \varrho\leq 1$ and 
\[
   m< n(\varrho-1)\brkt{\frac{1}{\min(2,q_{\rm max},p)}+\frac{1}{\min(2,q_{\rm min},\frac{p q_{\rm max} }{q_{\rm max}+p})}},
\]
with the convention that, if $p=\infty$ then $\frac{p q_{\rm max}}{q_{\rm max}+p}=q_{\rm max}$.
Then the bilinear $\Psi$DO
\begin{equation}\label{eq:bilinearPSDO}
	T_{a} (f,g)(x)= \iint a(x,\xi,\eta)\, e^{ i \p{x,\xi+\eta}} \hat{ f}(\xi) \hat{g}(\eta)\,\dd \xi\, \dd \eta,
\end{equation}
satisfies the estimate
\[
    \norm{T_a(f,g)}_{L^r}\leq C_{a,n} \norm{f}_{L^{q_1}}\norm{g}_{L^{q_2}},\quad \text{for every $f,\,g\in \S$.}
\]
\end{thm}
\begin{thm} Let $1\leq q_1,q_2\leq \infty$, $0<r<\infty$ with $\frac{1}r =\frac 1 {q_1}+\frac 1 {q_2}$.  Suppose that $a\in S^{m}_{\varrho,\delta}(n,2)$, $0<\varrho,\delta\leq 1$, $\delta<1$ and 
\[
	m<n(\varrho-1)\sqbrkt{\max\brkt{\abs{\frac{1}{2}-\frac{1}{q_1}},
	\abs{\frac{1}{2}-\frac{1}{q_2}}}+\frac{1}{\min(2,q_1,q_2)}}+\frac{n\min(\varrho-\delta,0)}{2},
\]
Then the associated bilinear pseudodifferential operator $T_a$ defined by \eqref{eq:bilinearPSDO}
satisfies 
\[
    \norm{T_a(f,g)}_{L^r}\lesssim \norm{f}_{L^{q_1}}\norm{g}_{L^{q_2}}. 
\]
\begin{proof} The proof is similar to that of Theorem \ref{thm:smoothbilinearFIO}, decomposing the amplitude in frequency regions and using \cite[Theorem 3.2]{AlvarezHounie} in the first step of the iteration argument and Theorem \ref{thm:LinearPDO} in the second step.
\end{proof}
\end{thm}
With the previous iteration argument we can obtain an improvement of \cite[Theorem 2]{BBMNT}.
\begin{thm} Let $1\leq q_1,q_2\leq \infty$, $0<r<\infty$ with $\frac{1}r =\frac 1 {q_1}+\frac 1 {q_2}$. 
Suppose that $a\in S^{m}_{\varrho,\delta}(n,2)$, $0<\varrho\leq 1$, $\delta\leq \varrho$, $\delta<1$ with 
\[
	m<n(\varrho-1)\sqbrkt{\max\brkt{\frac{1}{2},\,\frac{1}{q_1},\, \frac{1}{q_2},\, 1-\frac{1}{r}}+\frac{1}{2}\max\brkt{\frac{1}{r}-1,0}}.
\]
Then the associated bilinear pseudodifferential operator $T_a$ defined by \eqref{eq:bilinearPSDO}
satisfies the estimate
\[
    \norm{T_a(f,g)}_{L^r}\lesssim \norm{f}_{L^{q_1}}\norm{g}_{L^{q_2}},
\]

\begin{proof} We exclude the case $\varrho=1$ and $\delta<1$ since it lies in the realm of multilinear Calder\'on-Zygmund theory (see e.g. \cite{GrafTor}).

Observe that it suffices to prove the result for $\varrho=\delta<1$. Define on $S_{\varrho,\varrho}^m(n,2)\times L^{q_1}\times L^{q_2}$ the trilinear operator given by $T(a,f,g)=T_{a}(f,g)$. As a corollary of the previous result we have that,  $T$ satisfies
\begin{eqnarray*}
	T:& S^m_{\varrho,\varrho}(n,2)\times &L^{\infty}\times L^\infty \to L^\infty \quad  \text{for } \quad m<n(\varrho-1),\\
	T:& S^m_{\varrho,\varrho}(n,2)\times &L^{2}\times L^2 \to L^1 \quad \text{for } \quad m<\frac{n(\varrho-1)}{2},\\
	T:& S^m_{\varrho,\varrho}(n,2)\times &L^{1}\times L^1 \to L^{\frac{1}{2}} \quad \text{for } \quad m<\frac{3n(\varrho-1)}{2}.
\end{eqnarray*}
The first two, jointly with the symbolic calculus in \cite{BMNT} for amplitudes in $S^m_{\varrho,\varrho}(n,2)$ yield
\begin{eqnarray*}
	T:& S^m_{\varrho,\varrho}(n,2)\times &L^{1}\times L^\infty \to L^1 \quad  \text{for } \quad m<n(\varrho-1),\\
	T:& S^m_{\varrho,\varrho}(n,2)\times &L^{\infty}\times L^1 \to L^1 \quad  \text{for } \quad m<n(\varrho-1),\\
	T:& S^m_{\varrho,\varrho}(n,2)\times &L^{2}\times L^\infty \to L^2 \quad \text{for } \quad m<\frac{n(\varrho-1)}{2},\\
	T:& S^m_{\varrho,\varrho}(n,2)\times &L^{\infty}\times L^2 \to L^2 \quad \text{for } \quad m<\frac{n(\varrho-1)}{2}.
\end{eqnarray*}
The result follows by a trilinear complex interpolation argument, taking into account the interpolation properties of the class  $S^m_{\varrho,\varrho}(n,2)$ in \cite[Lemma 7]{BBMNT}.
\end{proof}
\end{thm}

\end{document}